\documentclass[a4paper,11pt]{article}
\usepackage{amsthm, amsmath, amssymb, dsfont, mathrsfs, color, upgreek, accents, graphicx}
\usepackage{float}
\usepackage{multirow}
\usepackage{hhline}
\usepackage[table,xcdraw]{xcolor}
\usepackage[utf8]{inputenc} 
\usepackage[T1]{fontenc}

\newtheorem{theorem}{Theorem}
\newtheorem{lemma}[theorem]{Lemma}
\newtheorem{proposition}[theorem]{Proposition}

\newtheorem{remark}[theorem]{Remark}
\newtheorem{claim}[theorem]{Claim}

\newcommand{\dbh}{\hat{\Upsilon}}

\begin{document}

\title{The contact process over a dynamical d-regular graph}
\author{Gabriel Leite Baptista da Silva\textsuperscript{1}, Roberto Imbuzeiro Oliveira\textsuperscript{2}, Daniel Valesin\textsuperscript{3}}
\footnotetext[1]{Bernoulli Institute, University of Groningen. g.da.silva@rug.nl}
\footnotetext[2]{Instituto Nacional de Matem\'atica Pura e Aplicada. rimfo@impa.br}
\footnotetext[3]{Bernoulli Institute, University of Groningen. d.rodrigues.valesin@rug.nl}

\maketitle

\begin{abstract} We consider the contact process on a dynamic graph defined as a random~$d$-regular graph with a stationary edge-switching dynamics. In this graph dynamics, independently of the contact process state, each pair~$\{e_1,e_2\}$ of edges of the graph is replaced by new edges~$\{e_1',e_2'\}$ in a crossing fashion: each of~$e_1',e_2'$ contains one vertex of~$e_1$ and one vertex of~$e_2$. As the number of vertices of the graph is taken to infinity, we scale the rate of switching in a way that any fixed edge is involved in a switching with a rate that approaches a limiting value~$\mathsf{v}$, so that locally the switching is seen in the same time scale as that of the contact process. We prove that if the infection rate of the contact process is above a threshold value~$\bar{\lambda}$ (depending on~$d$ and~$\mathsf{v}$), then the infection survives for a time that grows exponentially with the size of the graph. By proving that~$\bar{\lambda}$ is strictly smaller than the lower critical infection rate of the contact process on the infinite~$d$-regular tree, we show that there are values of~$\lambda$ for which the infection dies out in logarithmic time in the static graph but survives exponentially long in the dynamic graph.
\end{abstract} 


\section{Introduction} \label{s:intro}

\subsection{The contact process on finite graphs} The contact process is a class of spin systems that is usually taken as a simple model for the spread of an infection in a population. Vertices of a graph~$G = (V,E)$ can be infected (state~1) or healthy (state~0). The dynamics is given by the prescription that independently, infected vertices recover with rate~1, and healthy vertices become infected with rate~$\lambda$ times the number of infected neighbors, where~$\lambda > 0$ is the model parameter, called the infection rate.

The configuration in which all vertices are healthy is an absorbing state, and on finite graphs it is almost surely reached. A quantity of interest in this case is the hitting time of this configuration, for the process started from all vertices infected at time zero; this hitting time is called the \textit{extinction time} (of the infection), and is denoted by~$\tau_G$. Typically, one fixes the infection rate~$\lambda$, takes a sequence of growing graphs~$(G_n)_{n \ge 1}$ from some common model of interest, and studies the asymptotic behavior of~$\tau_{G_n}$ as~$n \to \infty$.  It turns out that, in several cases, this behavior changes drastically according to a threshold value of~$\lambda$; this change is referred to as a finite-volume phase transition, and can often be associated to a phase transition  of the contact process on a related infinite graph. For instance, let $\lambda_c(\mathbb{Z}^d)$ be the critical value of the contact process on~$\mathbb{Z}^d$, that is, the supremum of the parameter values for which the infection dies out almost surely in the  process on~$\mathbb{Z}^d$ started from a single infection. If~$G_n$ is a box of~$\mathbb{Z}^d$ with side length~$n$, then it is known that~$\tau_{G_n}$ grows logarithmically with~$n$ if~$\lambda < \lambda_c(\mathbb{Z}^d)$, and exponentially with~$n$ if~$\lambda > \lambda_c(\mathbb{Z}^d)$.  See~\cite{cgov,sc85,dl88,ds88,tom93,tom99}.

\subsection{The contact process on the random~$d$-regular graph} An instance of the finite-volume phase transition of the contact process that is of particular interest to us happens in the random~$d$-regular graph. Let us explain how this graph is constructed. Fix~$d \in \mathbb{N}$,~$d \ge 3$, and~$n \in \mathbb{N}$; assume that~$dn$ is even. Let~$V := \{1,\ldots, n\}$ be the set of vertices of the graph, and~$H:=V \times \{1,\ldots, d\}$ be the set of half-edges (we generally omit the dependence on~$d$ and~$n$). Sample uniformly at random a perfect matching~$\varphi:H \to H$ of the set of half-edges (that is~$\varphi: H \to H$ is a bijection satisfying~$\varphi^{-1} = \varphi$ and~$\varphi((x,a)) \neq (x,a)$ for all~$(x,a) \in H$), and regard all sets of the form~$\{(x,a),(x',a')\}$ with~$(x',h') = \varphi((x,h))$ as an edge of the graph. The set of edges is denoted~$E$. This gives rise to a graph~$G = (V,E)$ (actually, a multi-graph, since self-loops and parallel edges are allowed).

Assume that this random graph is sampled and the contact process with parameter~$\lambda > 0$ then evolves on it; in what follows, we will fix~$d,\lambda$ and take~$n \to \infty$ (along values so that~$dn$ is even). Let us clarify that self-loops have no effect in the contact process dynamics,  and parallel edges behave as separate media for the transmission of the infection (so, if for instance there are~$k$ edges between vertices~$x$ and~$y$, then an infection at~$x$ is transmitted to~$y$ with rate~$k \lambda$).

Independently in the two references~\cite{lalley,mv16}, the following result was proved. Let~$\mathbb{T}_d$ denote the infinite~$d$-regular tree, and let~$\lambda_c(\mathbb{T}_d)$ denote the supremum of the values of~$\lambda$ for which the contact process on~$\mathbb{T}_d$ started from a single infection dies out almost surely.

\begin{theorem}[\cite{lalley,mv16}]\label{thm_mv} For the contact process on the random~$d$-regular graph~$G_n$, we have that
	\begin{itemize}
		\item[$\mathrm{(a)}$] if~$\lambda < \lambda_c(\mathbb{T}_d)$, then there exists~$C > 0$ such that
			\[\mathbb{P}(\tau_{G} < C\log(n)) \xrightarrow{n \to \infty} 1;\]
		\item[(b)] if~$\lambda > \lambda_c(\mathbb{T}_d)$, then there exists~$c > 0$ such that
			\[\mathbb{P}(\tau_{G} > \exp\{c n\}) \xrightarrow{n \to \infty}1.\]
	\end{itemize}
\end{theorem}
The relevance of the threshold value~$\lambda_c(\mathbb{T}_d)$ comes from the fact that~$G$ rooted at a vertex chosen uniformly at random converges locally, in the sense of Benjamini and  Schramm~\cite{bs11}, to~$\mathbb{T}_d$ rooted at an arbitrary vertex.

\subsection{The contact process on the switching random~$d$-regular graph; main result} We now define an edge switching dynamics on the random~$d$-regular graph, a mechanism first introduced and studied in~\cite{cdg}. Let~$G=(V,E)$ be a realization of the random~$d$-regular graph on~$n$ vertices. Let~$e=\{(x,a),(y,b)\}$ and~$e'=\{(x',a'),(y',b')\}$ be two edges  of the graph, and assume that~$(x,a) < (y,b)$ and~$(x',a') < (y',b')$ in the lexicographic order of the set of half-edges~$H$. The \textit{switch with mark~$\mathsf{m} = (\{e,e'\},+)$} is the transformation that turns the graph~$G$ into the graph~$\Gamma^\mathsf{m}(G)$, which is equal to~$G$, except that the edges~$e,e'$ are removed and the two new edges~$\{(x,a),(x',a')\}$ and~$\{(y,b),(y',b')\}$ are added. We call this a \textit{positive switch}, as it makes a correspondence in accordance with the lexicographic order of half-edges (smaller with smaller, larger with larger). Similarly, the switch with mark~$\mathsf{n} = (\{e,e'\},-)$ is the transformation that turns~$G$ into~$\Gamma^\mathsf{n}(G)$, which is equal to~$G$ except that~$e,e'$ are removed and the edges~$\{(x,a),(y',b')\}$ and~$\{(x',a'),(y,b)\}$ are added; we call this a negative switch.

We now introduce a continuous-time Markov chain~$(G_t)_{t \ge 0}$ on the spaces of~$d$-regular graphs on~$n$ vertices as follows. We take~$G_0$ as a random~$d$-regular graph on~$n$ vertices chosen uniformly at random. Given the state~$G_t$ at time~$t$, for each of the~${dn/2 \choose 2}\cdot 2 = \frac{dn}{2}(\tfrac{dn}{2}-1)$ switch marks~$\mathsf{m}$ that can be formed from~$G_t$, we prescribe that the chain performs the jump~$G_t \to \Gamma^\mathsf{m}(G_t)$ with rate~$\frac{\mathsf{v}}{nd}$, where~$\mathsf{v} > 0$ is a parameter for the graph dynamics. It is readily seen that the uniform distribution on random~$d$-regular graphs on~$n$ vertices is stationary with respect to this  dynamics. The reason for the choice of rate~$\frac{\mathsf{v}}{nd}$ is that we want a fixed edge to be involved in a switch with a rate that is approximately equal, as~$n \to \infty$, to the parameter~$\mathsf{v}$. We call~$(G_t)_{t \ge 0}$ a \textit{switching graph with switch rate~$\mathsf{v}$}.

We will now consider the process~$(G_t,\xi_t)_{t \ge 0}$, where~$(G_t)$ is a switching random~$d$-regular graph on~$n$ vertices and~$(\xi_t)$ is a contact process with infection rate~$\lambda$. The definition of the contact process on the evolving graph is similar to that on the static one; a formal description is given in Section~\ref{sec_prelim}. As before, we start the contact process from all vertices infected at time zero, and study the extinction time, defined as the hitting time of the all-healthy configuration; here this time is denoted~$\tau_{(G_t)}$. Our main result is as follows.

\begin{theorem}
	\label{thm_main} Let~$d \ge 3$. For each~$\mathsf{v} > 0$ there exists~$\bar{\lambda}(\mathsf{v}) \in (0,\lambda_c(\mathbb{T}_d))$ (depending also on~$d$) such that the following holds. For any~$\lambda > \bar{\lambda}(\mathsf{v})$, there exists~$c > 0$ such that the extinction time of the contact process with infection rate~$\lambda$ on the switching random~$d$-regular graph~$(G_t)_{t \ge 0}$ with switch rate~$\mathsf{v}$ satisfies
	\[\mathbb{P}(\tau_{(G_t)} > \exp\{cn\}) \xrightarrow{n \to \infty} 1.\]
\end{theorem}

A highlight of this result is the fact that the long-term persistence of the infection on~$(G_t)$ holds for values of~$\lambda$ that are smaller than~$\lambda_c(\mathbb{T}_d)$. The process with such infection rates on the static version of the graph would reach extinction  quickly, by Theorem~\ref{thm_mv}. A rough intuitive explanation for this phenomenon is that the switchings can aid the spread of the infection: they introduce the possibility of separating a pair (transmitter, target) right after a transmission, possibly allowing both the transmitter and the target to now transmit the infection to their new neighbors, in case these new neighbors happen  to be healthy.

\subsection{Methods of proof and organization of paper}  The value~$\bar{\lambda}(\mathsf{v})$ that appears in Theorem~\ref{thm_main} is obtained as the critical value of an auxiliary process that is in a sense the local limit of the contact process on~$(G_t)$. We call this auxiliary process the  \textit{herds process}, as it consists of an evolving family of contact processes, all independent, each occupying its separate copy of the infinite tree~$\mathbb{T}_d$; each process in the family is called a \textit{herd}. Apart from the contact process evolution in each herd (which follows the usual rules of growth with rate~$\lambda$ and death with rate~$1$), herds can split. That is, for each herd and each edge that delimits two non-empty subsets of this herd, we take an exponential clock with rate~$\mathsf{v}$, and when this clock rings, the herd is replaced by two new herds, each containing one of the two aforementioned subsets. See Section~\ref{sec_basic} for a formal definition.

In that section, we start the study of the herds  process and state, in Theorem~\ref{thm_main}, that its critical value is  strictly smaller than~$\lambda_c(\mathbb{T}_d)$. The proof of this theorem is postponed to Section~\ref{sec_freeze}, to ease the flow of the exposition. The argument for this proof involves a coupling between the contact process on~$\mathbb{T}_d$, on the one hand, and the herds process, on the other hand, in a way that the former is stochastically dominated by the latter. We then show that, if~$\lambda$ is only slightly below~$\lambda_c(\mathbb{T}_d)$, then the mechanism of separation of transmitter and target described in the paragraph following Theorem~\ref{thm_main} occurs many times. This yields many occasions where nothing happens in the contact process, while new infections appear in the herds process. These extra particles can then be used to obtain a supercritical branching structure embedded inside the herds process.

In order to show that the contact process on~$(G_t)$ locally resembles the herds process, we first need to study a truncation of the latter, which we call the~$h$-herds process. This is done in Section~\ref{s_hherds}. In this alternate process, rather than evolving in~$\mathbb{T}_d$, herds evolve in finite subgraphs of~$\mathbb{T}_d$, each with diameter~$2h$. We argue that if the herds process is supercritical for a certain pair of parameter values~$(\lambda,\mathsf{v})$, then the~$h$-herds process with the same parameters and sufficiently large~$h$ is also supercritical. The~$h$-herds process can be regarded as a continuous-time multi-type branching process. Using this perspective, we take the associated Perron-Frobenius eigenvalue (which is larger than one in the supercritical regime) and associated eigenfunction (which is then a sub-harmonic function with respect to the dynamics of the~$h$-herds process).

Finally, in Section~\ref{s_embed}, we go back to~$(G_t,\xi_t)$, the contact process on the switching graph, showing how the results  from Section~\ref{sec_basic} and~\ref{s_hherds} lead to the proof of Theorem~\ref{thm_main}. We show that we can extract a collection of disjoint subsets of~$G_t$ (together with the infected vertices inside these subsets) and argue that the evolution of these subsets and the infection inside them closely resembles that of an~$h$-herds process. We use the sub-harmonic function mentioned above and martingale arguments to implement the comparison.

\subsection{Discussion and related works} As already mentioned, Theorem~\ref{thm_main} reveals an instance of metastability of the contact process. In forthcoming work,  the regime where~$(G_t,\xi_t)$ has parameter values~$(\mathsf{v},\lambda)$ with~$\lambda < \bar{\lambda}(\mathsf{v})$ will be studied, and it will be shown that fast extinction occurs in that case. This will complete the picture of a finite-volume phase transition.

Let us observe that the inequality of critical values in Theorem~\ref{thm_main} can be expressed as~$\bar{\lambda}(\mathsf{v}) < \bar{\lambda}(0)$ for~$\mathsf{v} > 0$, since the herds process with~$\mathsf{v} = 0$ is just a contact process on~$\mathbb{T}_d$. We conjecture that the function~$\mathsf{v} \mapsto \bar{\lambda}(\mathsf{v})$ is strictly decreasing on~$(0,\infty)$.

Apart from the aforementioned cases of lattice boxes and the random~$d$-regular graph, there are many works in the literature concerning finite-volume phase transitions of the contact process (on static graphs). See~\cite{cd09, mvy13} for the configuration model,~\cite{cd21,bnns21} for both the configuration model and the Erd\H{o}s-Renyi graph,~\cite{bbcs05,can17} for the preferential attachment graph, and~\cite{st01,cmmv} for truncated trees. There has also been recent progress on dynamic graph models; see~\cite{jm17,jlm19}.

\subsection{Some set and graph notation}
For any~$m \in \mathbb{N}$, we write~$[m]:= \{1,\ldots, m\}$. For a set~$A$, we denote by~$|A|$ the number of elements of~$A$.  We employ the usual abuse of notation of associating, for a set~$A$, a configuration~$\xi\in\{0,1\}^A$ with the set~$\{x \in A:\xi(x) = 1\}$.

In the rest of the paper,~$d \in \mathbb{N}$,~$d \ge 3$ will be kept fixed, and dependence on~$d$ will be omitted. In particular, we let~$\mathbb{T}$ denote the infinite~$d$-regular tree, with a distinguished root vertex~$o$. We sometimes abuse notation and use the same symbol to denote a graph and its set of vertices. 

We now present some of the graph notation we will employ. In Section~\ref{s_embed} our notation will need to accommodate to multi-graphs (where self-loops and parallel edges are allowed), but everywhere else in the paper we only deal with simple graphs, and in fact subgraphs of the infinite~$d$-regular tree. The notation we present here is intended for this simpler setting, and in Section~\ref{s_embed} we give the necessary additions.

Let~$G = (V,E)$ be a graph. We write~$x \sim y$ when vertices~$x$ and~$y$ are neighbors, and let~$\deg(x)$ denote the degree of~$x$. Let~$\mathrm{dist}(x,y)$ denote the graph distance between~$x$ and~$y$. When we wish to make the graph explicit, we write~$x \stackrel{G}{\sim}y$,~$\deg_G(x)$ and~$\mathrm{dist}_G(x,y)$. We let~$\mathcal{B}_G(x,r)$ denote the ball (in graph distance) with center~$x$ and radius~$r \in \mathbb{N}$.  The diameter of~$G$ is the maximum  attained by the graph distance between vertices of~$G$.

As already mentioned, we denote by~$\lambda_c(\mathbb{T})$ the supremum of the values of~$\lambda$ for which the contact process with rate~$\lambda$ on~$\mathbb{T}$ dies out almost surely, when started from finite configurations.

\section{The herds process} \label{sec_basic}
In this  section, we will define a Markov process whose state at a given time is an indexed family of finite subsets of the infinite~$d$-regular tree. Each of these finite sets is called a \textit{herd}. Herds evolve as independent contact processes, but can also split into two.  

The reason to introduce the herds process is that it arises naturally as a local limit of the contact process on a switching random~$d$-regular graph. Indeed, suppose that~$(G_t,\xi_t)_{t \ge 0}$ is defined as in the introduction, and that~$(\xi_t)$ starts with a single infection at a vertex chosen uniformly at random. Further suppose that we follow the dynamics of the infection ``within a fixed window'', that is, we watch the evolution of the set of infected vertices, but do not pay attention to regions of the graph that are free from infection. Then, apart from low-probability encounters with regions of the graph where loops are present, we would observe the contact process being naturally split into different ``islands'', with each island being further subdivided if one of its edges splits.

\subsection{Definition of the herds process}
For the rest of this section, we fix~$\mathsf{v} > 0$ and~$\lambda > 0$.  The herds  process will be denoted by~$(\Xi_t)_{t \ge 0}$; its state at a given time~$t$ is given by
\[
	\Xi_t = \left(\mathcal{J}_t,\{\eta^i_t:\;i \in \mathcal{J}_t \}\right),
\]
where~$\mathcal{J}_t$ is a finite set of indices (whose values are unimportant, but for concreteness we take~$\mathcal{J}_t \subset \mathbb{N}$) and for each~$i \in \mathcal{J}_t$,~$\eta^i_t$ is a finite subset of~$\mathbb{T}$. These subsets are called the herds at time~$t$. We say that an element of a herd is a \textit{particle}, and that particles can die and give birth. In this context, we deviate from the usual terminology involving infections, recoveries and transmissions, and instead say that  an element of a herd is a \textit{particle}, and that particles can die and give birth. We call the parameter~$\lambda$ a birth rate instead of an infection rate.

Let us first define~$(\Xi_t)_{t \ge 0}$ informally. Given a state~$\Xi_t = (\mathcal{J}_t,\{\eta^i_t\})$ at time~$t$, the chain evolves at times~$s \ge t$ as follows. Each of the herds~$\eta^i_s$ evolves independently as a contact process on~$\mathbb{T}$. Additionally, herds can split, as follows. Whenever an edge~$e$ delimits two non-empty portions of the herd (that is, the two connected components of~$\mathbb{T}$ obtained by the deletion of~$e$ intersect the herd), this edge is endowed with an exponential clock of rate~$\mathsf{v}$ (each clock is specific to a pair (edge, herd), and the clocks are all independent). When a clock rings, the herd is split, meaning that it is deleted and replaced by two new herds, each containing one of the two herd portions that were delimited by~$e$. The index set is adjusted according to these transitions: when a herd becomes empty (following the death of its last particle) its index is deleted, and when a herd splits, its index is replaced by two new indices corresponding to the new herds that replace it.

To give a formal definition, let us introduce some notation. Let~$\eta$ be a set of vertices of~$\mathbb{T}$ and~$e=\{u,v\}$ be an edge of~$\mathbb{T}$. The graph obtained  by removing~$e$ from~$\mathbb{T}$ has two connected components, one containing~$u$ and the other~$v$. Let~$\eta^{e,u}$ and~$\eta^{e,v}$ denote the intersection of~$\eta$ with the corresponding components.

Now, in order to define the continuous-time Markov chain~$(\Xi_t)_{t \ge 0}$, it suffices to specify all kinds and rates of jumps  the chain can perform from a fixed state~$\Xi = (\mathcal{J},\;\{\eta^i_t:i \in \mathcal{J}\})$ (we will also show non-explosiveness shortly). They are as follows:
\begin{itemize}
	\item contact birth -- for each~$i \in \mathcal{J}$,~$u \in \eta^i$ and~$v \sim u$, with rate~$\lambda$: the process jumps from~$\Xi$ to the state in which the index set~$\mathcal{J}$ is kept the same as in~$\Xi$ and the herds with index~$j \neq i$ are kept the same as in~$\Xi$, while herd~$\eta^i$ is replaced by~$\eta^i\cup\{v\}$;
	\item contact death, with removal of empty herds -- for each~$i \in \mathcal{J}$ and each~$x\in\eta^i$, with rate one: the process jumps from~$\Xi$ to the state~$\Xi'$ defined as follows. In case~$\eta^i = \{x\}$, then herd~$i$ is simply deleted: the index set of~$\Xi'$ is~$\mathcal{J}\backslash\{i\}$, and all other herds are left unchanged. In case~$\eta^i \neq \{x\}$, then~$\Xi'$ has the same index set~$\mathcal{J}$ as~$\Xi$; the herds with index~$j \neq i$ are kept the same as in~$\Xi$, while herd~$\eta^i$ is replaced by~$\eta^i\backslash \{x\}$;
	\item herd splitting: for each~$i \in \mathcal{J}$ and each edge~$e=\{u,v\}$ for which~$(\eta^i)^{e,u}$ and~$(\eta^i)^{e,v}$ are both non-empty, with rate~$\mathsf{v}$: the process jumps from~$\Xi$ to the state~$\Xi'$ defined as follows. The index set of~$\Xi'$ is~$\mathcal{J}':=(\mathcal{J}\backslash \{i\})\cup\{i_1,i_2\}$, where~$i_1,i_2$ are arbitrary natural numbers not belonging to~$\mathcal{J}$. All herds~$\eta^j$ with~$j \neq i$ are unchanged, and~$\eta^{i_1}:= (\eta^i)^{e,u}$ and~$\eta^{i_2}:=(\eta^i)^{e,v}$.
\end{itemize}

Unless we explicitly mention otherwise, we will assume that the herds process is started from a single herd with a single particle at time zero. We also emphasize that we will never consider this process started from a configuration with either infinitely many herds or with one or more herds with infinitely many particles. 

We now make three observations about the herds process.
\begin{itemize}
	\item[(a)] \textit{Non-explosiveness.} The following is a brief argument to show that the herds process  almost surely performs finitely many jumps in finite time intervals. Let~$X_t$ denote the number of times in~$[0,t]$ the process~$(\Xi_t)$ has performed a jump of either the ``contact birth'' or ``contact death'' types.  Then,~$(X_t)_{t \ge 0}$ is stochastically dominated by a  continuous-time pure-birth process~$(Z_t)_{t \ge 0}$ on~$\mathbb{N}$ that jumps from~$m$ to~$m+1$ with rate~$(d\lambda+1) m$. Since~$(Z_t)$ is non-explosive, so is~$(X_t)$. Next, note that between any two jumps of~$(X_t)$, there is a maximum number of split-type transitions that can occur in~$(\Xi_t)$ (until the point is reached when all particles are isolated in a herd and no more splits can happen). This concludes the proof.
	\item[(b)] \textit{Genealogy of herds.} By keeping track of a parent-child relation between herds whenever there is a split, we naturally obtain a genealogical relation between herds along time. That is,  the set of herds at any time~$t$ could be partitioned according to the herd at some earlier time~$s$ they descend from. We refrain from introducing notation in this direction for the sake of simplicity, but will occasionally refer to the genealogical structure in our proofs.
	\item[(c)] \textit{Survival.} We say that the herds process dies if there exists some time~$t$ at which the index set is empty (due to the death of the last herd at some earlier time); we then write~$\Xi_t = \varnothing$ (and we evidently have~$\Xi_s = \varnothing$ for all~$s \ge t$). In the event that this does not hold for any~$t$, we say that the process survives. Using elementary irreducibility considerations, it is not hard to see that the survival probability is either positive for any non-empty initial configuration or zero for any  initial configuration (since we only take finite initial configurations).
\end{itemize}
In light of the last comment, we define
\begin{equation}\label{eq_def_of_bar_lambda}
	\bar{\lambda}(\mathsf{v}):= \sup\{\lambda:\;\mathbb{P}_{\lambda,\mathsf{v}}((\Xi_t)\text{ dies})= 1\}, 
\end{equation}
where~$\mathbb{P}_{\lambda,\mathsf{v}}$ denotes a probability measure under which~$(\Xi_t)$ with birth rate~$\lambda$ and split rate~$\mathsf{v}$ is defined.
 
The following strict inequality between critical rates is a fundamental ingredient in proving Theorem~\ref{thm_main}, and is also of independent interest.  We postpone the proof to Section~\ref{sec_freeze}. 
\begin{theorem}\label{thm_strict}
	For any~$\mathsf{v} > 0$ we have~$\bar{\lambda}(\mathsf{v}) < \lambda_c(\mathbb{T})$.
\end{theorem}

The following simple fact will also be useful in the next section.
\begin{lemma}\label{lem_for_herds}
	\label{lem_Xi_large} Assume that~$\mathsf{v} >0$ and~$\lambda > \bar{\lambda}(\mathsf{v})$. Then, the herds process~$(\Xi_t) = (\mathcal{J}_t,\;\{\eta^i_t\})$ satisfies
	\[\mathbb{P}\left( |\mathcal{J}_t|\xrightarrow{t \to \infty} \infty \mid (\Xi_t) \text{ survives}\right) = 1.\]
\end{lemma}

\begin{proof}
	Since the proof involves standard arguments for Markov chains, we only sketch it.

	Fix~$k > 0$. Let~$\sigma$ denote the first time~$t$ at which there are~$k$ distinct indices~$i_1,\ldots,i_k \in \mathcal{J}_{t}$ such that the herds~$\eta^{i_1}_{t},\ldots,\eta^{i_k}_{t}$ are all singletons. We have that~$\sigma < \infty$ almost surely conditioned on~$\{(\Xi_t) \text{ survives}\}$, since it is not hard to see that there exists~$q = q(k) > 0$ such that, from any non-empty configuration at time~$t$, the process has probability at least~$q$ of producing~$k$ new herds that are singletons, and remain this way, until time~$t+1$.
	
	On the event~$\{\sigma < \infty\}$, fix~$k$ indices~$i_1,\ldots, i_k \in \mathcal{J}_\sigma$ such that~$\eta^{i_j}_\sigma$ is a singleton for each~$j$, and then make a trial, where a success is defined as the survival of all the~$k$ lineages started from the herds represented by~$i_1,\ldots, i_k$ at time~$\sigma$. The probability of success is~$\rho^k$, where~$\rho$ is the survival probability of the herds process started from a single singleton herd. In case there is a failure, we let~$\tilde{\sigma}$ denote the first death time of one of the~$k$ lineages involved in the trial, and then we start again after~$\tilde{\sigma}$. That is, we let~$\sigma^{(1)}$ denote the first time after~$\tilde{\sigma}$ at which there are at least~$k$ singleton herds, run a new trial etc. Conditioned on~$\{(\Xi_t) \text{ survives}\}$, a success eventually occurs almost surely, and after the starting time of the successful trial, there are always at least~$k$ herds in the process. Since~$k$ is arbitrary, this completes the proof.
\end{proof}

\section{The~$h$-herds process}\label{s_hherds}
As already explained, we would like to argue that the contact process on the switching random~$d$-regular graph resembles the herds process. An intermediate step in this direction is to truncate the herds process, so that herds only occupy finite subsets of~$\mathbb{T}$, with the idea that these subsets can then be isomorphically embedded in~$G_t$. 

A first attempt for such a truncation would be to prescribe that herds can only occupy the set~$\mathcal{B}_\mathbb{T}(o,h)$ for some  large~$h$, and that when there is a splitting of an edge of this ball, each of the two resulting components are augmented so as to restore the piece that was severed, thus making them again isomorphic to the same ball. However, it turns out that this definition would not be appropriate in an important respect, which we now explain.

Part of our strategy involves arguing that, if the herds process is supercritical for some parameters~$\lambda, \mathsf{v}$, then the truncated herds process with the same parameters and sufficiently long truncation range~$h$ is also supercritical. In proving this, one is naturally led to consider the multi-type branching structure of the truncated process. In trying to argue that this branching process survives, it is useful to appeal to irreducibility-like properties, for instance that from any given herd it is possible to generate a herd of the same shape as the initial one (with a single particle at the root) within one time unit with a probability that does not depend on~$h$. This is however not satisfied with the splitting scheme described in the previous paragraph: if a herd only has particles near the leaves of~$\mathcal{B}_\mathbb{T}(o,h)$, then it is costly (in an~$h$-dependent way) to produce the initial herd again. To overcome this difficulty, we propose an alternate splitting scheme that makes this irreducibility property more attainable.

\subsection{More tree notation: splitting trees}\label{ss_splitting_trees}
Let~$A$ be a subtree of~$\mathbb{T}$ (here and in what follows, whenever we refer to a subtree of~$\mathbb{T}$, we assume that it is connected), and let~$e=\{u,v\}$ be an edge of~$A$. We will now introduce some subgraphs of~$\mathbb{T}$ that can be defined  from~$A$ and~$e$.

First, the removal of~$e$ from~$A$ breaks~$A$ into two components, one containing~$u$ and the other containing~$v$; these are denoted by~$\mathcal{L}(A,e,u)$ and~$\mathcal{L}(A,e,v)$ respectively.

Next, for each~$r \ge 1$, let~$\mathcal{L}_r(A,e,u)$ denote  the (connected) subgraph of~$\mathbb{T}$ obtained by joining
\[\mathcal{L}(A,e,u),\qquad \mathcal{L}(\mathcal{B}_\mathbb{T}(v,r-1),e,v),\qquad \text{ and the edge~$e$}.\] Equivalently,~$\mathcal{L}_r(A,e,u)$ is the subgraph of~$\mathbb{T}$ induced by the set of vertices obtained as the union of the vertices of~$\mathcal{L}(A,e,u)$ with the set of all vertices that can be reached by a self-avoiding path in~$\mathbb{T}$ that starts at~$u$, first moves to~$v$, and then moves at most~$r-1$ steps. See Figure~\ref{fig:splitting} for an illustration.

\begin{figure}[htb]
	\begin{center}
		\setlength\fboxsep{0pt}
		\setlength\fboxrule{0.8pt}
		\fbox{\includegraphics[width = \textwidth]{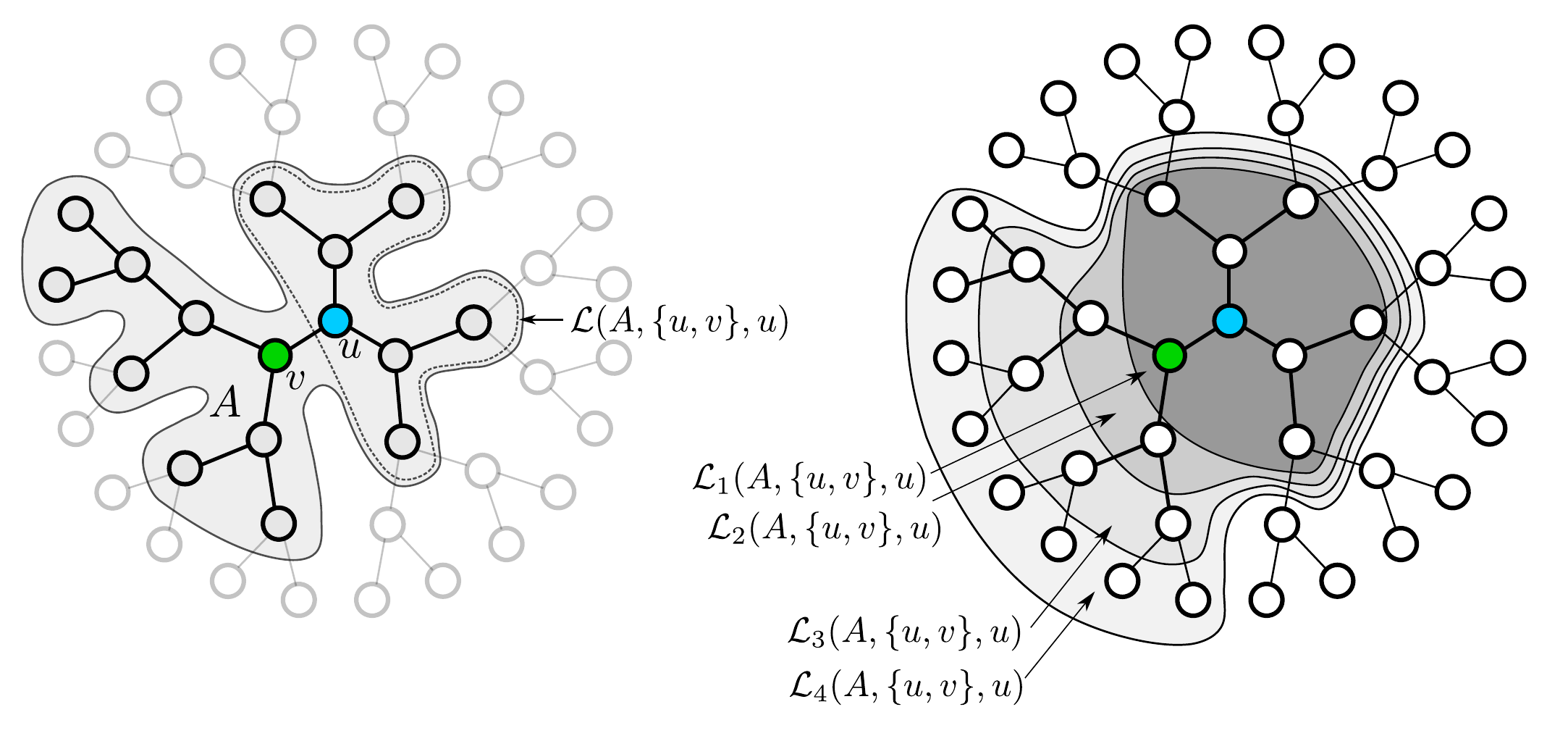}}
	\end{center}
	\caption{Illustration of the sets~$\mathcal{L}(A,\{u,v\},u)$ and~$\mathcal{L}_r(A,\{u,v\},u)$.}
	\label{fig:splitting}
\end{figure}

Now fix~$h \ge 1$ and assume that the diameter of~$A$ is at most~$2h$. Define
\begin{equation}\label{eq_def_of_r*}\begin{split}&\mathcal{T}_h(A,e,u) := \mathcal{L}_{r_*}(A,e,u),\\[.2cm]
&\hspace{1cm}\text{where }r_* := \max\{r \le h:\; \mathcal{L}_r(A,e,u) \text{ has diameter at most $2h$}\}.\end{split}\end{equation}
The production of the two trees~$\mathcal{T}_h(A,e,u)$ and~$\mathcal{T}_h(A,e,v)$ from~$A$ is called the \textit{$h$-splitting of~$A$ through~$e$}. In what follows, we will be interested in this operation, starting with the case where~$A = \mathcal{B}_\mathbb{T}(o,h)$ (which of course has diameter~$2h$).

Let~$\mathscr{A}_h$ denote the set of all subtrees of~$\mathbb{T}$ that can be reached from~$\mathcal{B}_\mathbb{T}(o,h)$ by performing successive~$h$-splittings. In other words,~$\mathscr{A}_h$ is the set of subtrees~$A$ of~$\mathbb{T}$ such that there exists a sequence~$A_0,A_1,\ldots,A_n$ of subtrees of~$\mathbb{T}$ with~$A_0 = \mathcal{B}_\mathbb{T}(o,h)$,~$A_n = A$ and so that for each~$i \in \{1,\ldots, n\}$,~$A_{i}$ is one of the two trees obtained by an~$h$-splitting of~$A_{i-1}$ through one of its edges. Note that~$\mathcal{B}_{\mathbb{T}}(o,h) \in \mathscr{A}_h$.

Let us observe that 
\begin{equation}\label{eq_deg_one_or_d} A \in \mathscr{A}_h \quad\Longrightarrow \quad A \text{ has diameter $2h$ and each vertex~$u$ of~$A$ has~$\deg_A(u) \in \{1,d\}$}.\end{equation} 
	Indeed, these properties  are satisfied by~$A = B_\mathbb{T}(o,h)$ and are preserved by the~$h$-splitting operation.

Given~$A \in \mathscr{A}$, let~$\partial A$ denote the set of leaves (vertices of degree one) of~$A$. In the proof of the  following lemma, it will be useful to note that
\begin{equation}\label{eq_dist_less_h}
	A \in \mathscr{A}, \; u \in A \quad \Longrightarrow \quad \mathrm{dist}_A(u,\partial A) \le h.
\end{equation}
This follows from noting that~$\mathcal{B}_\mathbb{T}(u,\mathrm{dist}_A(u,\partial A)) \subset A$ by~\eqref{eq_deg_one_or_d}, so
\[2\cdot \mathrm{dist}_A(u,\partial A) = \text{diameter of } \mathcal{B}_\mathbb{T}(u,\mathrm{dist}_A(u,\partial A)) \le \text{diameter of } A \le 2h.\]

We now give a lemma that will be useful in comparing the~$h$-herds process with the herds process.
\begin{lemma}\label{lem_reduce_dist}
	Let~$A \in \mathscr{A}$,~$e = \{u,v\}$ be an edge of~$A$ and~$w$ be a vertex of~$\mathcal{L}(A,e,u)$. Writing~$A':= \mathcal{T}_h(A,e,u)$, we have~$\mathrm{dist}_A(w,\partial A) \le \mathrm{dist}_{A'}(w,\partial A')$.
\end{lemma}
\begin{proof}
Define
	\[L_u := \partial A \cap \mathcal{L}(A,e,u) = \partial A' \cap \mathcal{L}(A',e,u)\]
	and
	\[L_v:= \partial A \cap \mathcal{L}(A,e,v),\qquad L_v':= \partial A' \cap \mathcal{L}(A',e,v),\]
	so that~$L_u$,~$L_v$ partition~$\partial A$ and~$L_u$,~$L_v'$ partition~$\partial A'$.
Also let
	\[r_0:= \mathrm{dist}_A(u,L_v) \quad \text{and}\quad r_*:= \mathrm{dist}_{A'}(u,L_v');\]
	note that~$r_*$ is as in~\eqref{eq_def_of_r*}. Also using~\eqref{eq_dist_less_h}, we now have
	\begin{align}
		&h\ge \mathrm{dist}_A(w,\partial A) = \min\{\mathrm{dist}_A(w,L_u),\; \mathrm{dist}_A(w,u)+ r_0\},\label{eq_dist_one}\\
		&h \ge \mathrm{dist}_{A'}(w,\partial A') = \min\{\mathrm{dist}_A(w,L_u),\; \mathrm{dist}_A(w,u) + r_*\}.\label{eq_dist_two}
	\end{align}

	Next, note that the definition of~$h$-splitting in~\eqref{eq_def_of_r*} gives~$r_* \le h$ and moreover, if~$\tilde{r}$ is such that~$\tilde{r} \le \min\{r_0,h\}$, then~$\tilde{r} \le r_*$. Hence,
	\begin{equation}\label{eq_dist_three}
		\min\{r_0,h\} \le r_* \le h.
	\end{equation}
	The desired inequality~$\mathrm{dist}_A(w,\partial A) \le \mathrm{dist}_{A'}(w,\partial A')$ now follows from~\eqref{eq_dist_one},~\eqref{eq_dist_two} and~\eqref{eq_dist_three}, by separately considering the cases~$r_0 > h$ and~$r_0 \le h$.
\end{proof}

The following lemma shows that it is possible to re-obtain a ball of radius~$h$ in~$\mathbb{T}$ by starting from an element of~$\mathscr{A}_h$ and performing at most~$d$ successive~$h$-splittings. The proof is straightforward and left to the reader.
\begin{lemma}\label{lem_scheme}
	Let~$A \in \mathscr{A}_h$ and~$u$ be a vertex of~$A$ with~$\deg_A(u) = d$. Let~$v_1,\ldots, v_d$ be the neighbors of~$u$, in decreasing order of depth with respect to~$u$, meaning that
	\[\{1,\ldots, d\} \ni i \mapsto \max\{\mathrm{dist}_A(u,v): \; v \in \mathcal{L}(A,\{u,v_i\},v_i)\} \]
	is non-increasing.
 Then, letting~$A_0 = A$ and recursively letting
	\[ A_{i+1} = \mathcal{T}_h(A_i,\{u,v_i\},u),\qquad i \in \{0,\ldots, d-1\},\]
	we have that~$A_d = \mathcal{B}_\mathbb{T}(u,h)$.
\end{lemma}

\subsection{Definition of the~$h$-herds process}\label{ss_def_Phi}
Define
\[\mathscr{A}_h' := \{(A,\alpha): A \in \mathscr{A}_h,\; \alpha \text{ is a subset of the set of vertices of~$A$}\}.\]
A pair~$(A,\alpha) \in \mathscr{A}_h'$ is called an~\textit{$h$-herd}. We interpret~$\alpha$ as a set of living particles and~$A$ as the region of space that these particles (and their descendants) can currently occupy. Given an~$h$-herd~$(A,\alpha)$, we say that an edge~$e= \{u,v\}$ of~$A$ is \textit{active} if~$e$ delimits two non-empty portions of~$\alpha$. We then define
\begin{align*}&\mathcal{T}_h((A,\alpha),e,u):= (\mathcal{T}_h(A,e,u),\;\alpha \cap \mathcal{L}(A,e,u)),\\
&\mathcal{T}_h((A,\alpha),e,v):= (\mathcal{T}_h(A,e,v),\;\alpha \cap \mathcal{L}(A,e,v)),\end{align*} 
and we refer to the mapping of~$(A,\alpha)$ into the two~$h$-herds above as an \textit{$h$-splitting} (or just \textit{splitting}, when~$h$ is clear from the context) of~$(A,\alpha)$ through~$e$ (which needs to be an active edge, a requirement that depends on~$\alpha$).

We will now define the \textit{$h$-herds process}
\[\Phi_t = (\mathcal{J}_t,\; \{(A^i_t,\alpha^i_t): i \in \mathcal{J}_t\}),\quad t \ge 0,\]
where for each~$t$,~$\mathcal{J}_t \subset \mathbb{N}$ is a finite set of indices and the pairs~$(A^i_t,\alpha^i_t)$ are~$h$-herds. This process will be similar to the herds process, with the sets~$\alpha^i_t$ here playing the roles of the sets~$\eta^i_t$ there, with two important differences. First, given~$i \in \mathcal{J}_t$, the particles corresponding to points of~$\alpha^i_t$ are only allowed to give birth on vertices of~$A^i_t$. Second, the splitting of~$h$-herds occurs according to the procedure described above, replacing a pair~$(A^i_t,\alpha^i_t)$ by the two~$h$-herds~$\mathcal{T}_h((A^i_t,\alpha^i_t),e,u),\mathcal{T}_h((A^i_t,\alpha^i_t),e,v)$. 

Let us now give a formal definition of the process. Again, it suffices to describe the jumps and rates from a state~$\Phi = (\mathcal{J},\{(A^i,\alpha^i):i \in \mathcal{J}\})$. These are:
\begin{itemize}
	\item contact birth -- for each~$i \in \mathcal{J}$,~$u \in \alpha^i$ and~$v \in A^i$,~$v \sim u$, with rate~$\lambda$:  the process jumps from~$\Phi$ to the state in which the index set~$\mathcal{J}$ is  the same as in~$\Phi$ and the~$h$-herds with index~$j \neq i$ are  the same as in~$\Phi$, while the~$h$-herd~$(A^i,\alpha^i)$ is replaced by~$(A^i,\alpha^i \cup \{v\})$;
	\item contact death, with removal of empty herds -- for each~$i \in \mathcal{J}$ and~$u\in\alpha^i$, with rate one: the process jumps from~$\Phi$ to the state~$\Phi'$ defined as follows. In case~$\alpha^i = \{u\}$, then herd~$i$ is simply deleted: the index set of~$\Phi'$ is~$\mathcal{J}\backslash\{i\}$, and all other herds are left unchanged. In case~$\alpha^i \neq \{u\}$, then~$\Phi'$ has the same index set~$\mathcal{J}$ as~$\Phi$; the herds with index~$j \neq i$ are kept the same as in~$\Phi$, while herd~$\alpha^i$ is replaced by~$\alpha^i\backslash \{u\}$;	
	\item $h$-herd splitting -- for each~$i \in \mathcal{J}$ and each edge~$e=\{u,v\}$ of~$A^i$ that is active (according to~$\alpha^i$), with rate~$\mathsf{v}$: the process jumps from~$\Phi$ to the state~$\Phi'$ defined as follows. The index set of~$\Phi'$ is~$(\mathcal{J}\backslash \{i\})\cup\{i_1,i_2\}$, where~$i_1,i_2$ are arbitrary natural numbers not belonging to~$\mathcal{J}$. All herds~$(A^j,\alpha^j)$ with~$j \in \mathcal{J} \backslash  \{ i\}$ are unchanged, while
		\[(A^{i_1},\alpha^{i_1}) = \mathcal{T}_h((A^i,\alpha^i),e,u)\qquad \text{and}\qquad (A^{i_2},\alpha^{i_2}) = \mathcal{T}_h((A^i,\alpha^i),e,v).\] 
\end{itemize}

As was the case for~$(\Xi_t)$, the set of~$h$-herds of~$(\Phi_t)$ is always assumed to be finite at~$t=0$, and almost surely it remains finite for all times. Unless explicitly stated otherwise, we assume that~$\Phi_0$ consists of a single~$h$-herd~$(\mathcal{B}_\mathbb{T}(o,h),\{o\})$.

The following lemma will be obtained as a consequence of Lemma~\ref{lem_scheme}.

\begin{lemma}[Uniform irreducibility bound]
	\label{lem_regardless_of}
	For any~$\lambda>0$ and~$\mathsf{v} > 0$, there exist~$p_0 > 0$ and~$t_0>0$ such that the following holds for any~$h \in \mathbb{N}$. Assume that the~$h$-herds process~$(\Phi_t)$ is started from an arbitrary non-empty state. Then, with probability at least~$p_0$, at time~$t_0$ there exists some~$i \in \mathcal{J}_{t_0}$ such that~$(A^i_{t_0},\alpha^i_{t_0}) = (\mathcal{B}_\mathbb{T}(u,h),\{u\})$ for some~$u \in \mathbb{T}$.
\end{lemma}
We postpone the proof of Lemma~\ref{lem_regardless_of} to Section~\ref{ss_proof_regardless}.  For now, let us see an important application. Say that~$(\Phi_t)$ survives if the event~$\{\mathcal{J}_t \neq \varnothing \text{ for all } t\}$ occurs.

\begin{lemma}
	\label{lem_regardless_of0}
	For any~$\lambda > 0$,~$\mathsf{v} > 0$ and~$h \in \mathbb{N}$, the 
	probability of survival of~$(\Phi_t)$ is either zero for any  initial state, or it is positive for any non-empty initial state. \end{lemma}
\begin{proof}
	It suffices to apply Lemma~\ref{lem_regardless_of} together with the observation that the~$h$-herds process started from a single~$h$-herd~$(\mathcal{B}_\mathbb{T}(u,h),\{u\})$ has the same distribution, modulo applying a tree translation to all~$h$-herds, as the~$h$-herds process started from a single~$h$-herd~$(\mathcal{B}_\mathbb{T}(o,h),\{o\})$. 
\end{proof}

Next we have the following result, showing that survival of the herds process implies survival of the~$h$-herds process with sufficiently large~$h$. Recall the definition of~$\bar{\lambda}(\mathsf{v})$ from~\eqref{eq_def_of_bar_lambda}.
\begin{lemma}[Truncation]\label{lem_manyo}
	Let~$\mathsf{v} > 0$ and~$\lambda > \bar{\lambda}(\mathsf{v})$. Then, for any~$k > 0$ there exist~$h_1 \in\mathbb{N}$ and~$t_1 > 0$ such that the  following holds for any~$h \ge h_1$. The~$h$-herds process~$(\Phi_t)$ with parameters~$\lambda,\mathsf{v},h$ started from a single~$h$-herd~$(\mathcal{B}_\mathbb{T}(o,h),\{o\})$ satisfies
	\[\mathbb{P}(|\mathcal{J}_{t_1}| \ge k) > \rho/2,\]
	where~$\rho$ is the survival probability of the herds  process~$(\Xi_t)$ with parameters~$\lambda,\mathsf{v}$ started from a single herd~$\{o\}$.
\end{lemma}
We postpone the proof of this lemma to Section~\ref{ss_proof_manyo}. 

Define~$\bar{\lambda}(\mathsf{v},h)$ as the infimum of the values of~$\lambda$ for which the~$h$-herds process with split rate~$\mathsf{v}$ and birth rate~$\lambda$ survives with positive probability. 
\begin{proposition}\label{prop_trunc}
	For any~$\mathsf{v} > 0$ we have~${\displaystyle \limsup_{h \to \infty} \bar{\lambda}(\mathsf{v},h) \le \bar{\lambda}(\mathsf{v}) }$.
\end{proposition}

\begin{remark}
	 Although we will not need this, it is worth mentioning that a simple coupling shows that~$\bar{\lambda}(\mathsf{v},h) \ge \bar{\lambda}(\mathsf{v})$ for all~$h \ge 0$. Together with the above proposition, this shows that~${\displaystyle \lim_{h \to \infty}\bar{\lambda}(\mathsf{v},h) = \bar{\lambda}(\mathsf{v})}$.
\end{remark}
\begin{proof}[Proof of Proposition~\ref{prop_trunc}]
	Fix~$\mathsf{v} > 0$ and~$\lambda > \bar{\lambda}(\mathsf{v})$. We will show that if~$h$ is large enough, then the~$h$-herds process with parameters~$\mathsf{v}$,~$\lambda$,~$h$ survives with positive probability.

	Let~$p_0$ and~$t_0$ be as in Lemma~$\ref{lem_regardless_of}$. Let~$k = \frac{4}{p_0\cdot \rho}$, where~$\rho$ is the survival probability of the herds  process~$(\Xi_t)$ with parameters~$\lambda,\mathsf{v}$ and started from a single herd~$\{o\}$. Choose~$t_1$ and~$h_1$ corresponding to~$k$ in Lemma~\ref{lem_manyo}.

	Now assume that~$h \ge h_1$. For the~$h$-herds process~$(\Phi_t)$ with parameters~$\mathsf{v}$,~$\lambda$,~$h$, for each~$t \ge 0$ let~$N_t$ denote the number of~$h$-herds at time~$t$ such that there exists~$u \in \mathbb{T}$ such that the~$h$-herd is of the form~$(\mathcal{B}_\mathbb{T}(u,h),\{u\})$. Using Lemmas~\ref{lem_regardless_of} and~\ref{lem_manyo} and the Markov property, we have that
	\[\mathbb{E}[N_{t+t_0+t_1}\mid \Phi_t] \ge N_t \cdot k \cdot \frac{\rho}{2} \cdot p_0 \ge 2N_t.\]  
	Indeed, each~$h$-herd of the form~$(\mathcal{B}_\mathbb{T}(u,h),\{u\})$ at time~$t$ has probability at least~$\rho/2$ of producing a lineage with at least~$k$~$h$-herds by time~$t+t_1$, and each of these~$h$-herds has probability at least~$p_0$ of producing an~$h$-herd of the form~$(\mathcal{B}_\mathbb{T}(u,h),\{u\})$ by time~$t+t_0+t_1$.

 This shows that the process~$\{N_{m\cdot (t_0+t_1)}: m \in \mathbb{N}_0\}$ is a supercritical branching process, so it survives with positive probability, which also implies that~$(\Phi_t)$ survives with positive probability.
	
\end{proof}

\subsection{Truncation: proof of Lemma~\ref{lem_manyo}} \label{ss_proof_manyo}

For the~$h$-herds process~$(\Phi_t)$, let~$\tau_\mathrm{leaf}$ be the first time at which there exists a herd with a particle at a leaf location, that is,
\begin{equation}\label{eq_def_leaf}\tau_\mathrm{leaf}:=\inf\left\{t\ge 0:\; \alpha^i_t \cap \partial A^i_t \neq \varnothing \text{ for some $i \in \mathcal{J}_t$}\right\}.\end{equation}

	\begin{lemma}\label{lem_leaf}
		Let~$\mathsf{v} > 0$,~$\lambda >0$,~$s > 0$, and~$\varepsilon > 0$. There exists~$h_2 \in \mathbb{N}$ such that, if~$h \ge h_2$ and~$(\Phi_t)$ is the~$h$-herds process with parameters~$\mathsf{v},\lambda,h$ started from a single herd~$(\mathcal{B}_\mathbb{T}(o,h),\{o\})$, then~$\tau_\mathrm{leaf} > s$ with probability larger than~$1-\varepsilon$.
\end{lemma}

\begin{proof}
Define
	\[D_t:= \min_{i \in \mathcal{J}_t}\; \mathrm{dist}_{A^i_t}(\alpha^i_t,\partial A^i_t),\qquad t \ge 0.\]
	That is,~$D_t$ denotes the minimum distance between a particle and a leaf vertex, among all~$h$-herds.
	Note that~$D_t$ does not decrease as the result of contact deaths, and decreases by  at most one when there is a contact birth. Lemma~\ref{lem_reduce_dist} implies that~$D_t$ does not decrease as the result of~$h$-splittings. Since $D_0 = h$,~$\tau_\mathrm{leaf}$ does not happen before the occurrence of the~$h$-th contact birth of the dynamics of~$(\Phi_t)$. The statement of the lemma now follows by noting that the number of contact births is stochastically dominated by a pure-birth process that jumps from~$m$ to~$m+1$ with rate~$d\lambda m$, and is hence independent of~$h$.
\end{proof}

\begin{proof}[Proof of Lemma~\ref{lem_manyo}]
	Fix~$\mathsf{v}>0$ and~$\lambda > \bar{\lambda}(\mathsf{v})$. Given~$h \in \mathbb{N}$, we will construct a coupling~$(\hat{\Xi}_t,\hat{\Phi}_t)$ between a herds process~$(\hat{\Xi}_t)=(\hat{\mathcal{J}}^\Xi,\{\eta^i_t\})$ with parameters~$\lambda$,~$\mathsf{v}$ and an~$h$-herds process~$(\hat{\Phi}_t)=(\hat{\mathcal{J}}^\Phi,\{(A^i_t,\alpha^i_t)\})$ with parameters~$\lambda$,~$\mathsf{v}$,~$h$. To do this, we start with a probability space in which~$(\hat{\Phi}_t)$ is defined, started from a single~$h$-herd of the form~$(B_\mathbb{T}(o,h),\{o\})$. We define~$\hat{\tau}_\mathrm{leaf}$ as in~\eqref{eq_def_leaf}, as the first time at which~$(\hat{\Phi}_t)$ has an~$h$-herd in which a leaf vertex  is occupied by a particle. Then, enlarging the probability space, we define~$(\hat{\Xi}_t)$ as follows. First, for~$t \le \hat{\tau}_\mathrm{leaf}$, we set~$\hat{\mathcal{J}}^{\Xi}_t = \hat{\mathcal{J}}^\Phi_t$ and~$\eta^i_t = \alpha^i_t$ for all~$i \in \hat{\mathcal{J}}^\Phi_t$. Next, on~$\{\hat{\tau}_\mathrm{leaf}<\infty\}$, we let~$(\hat{\Xi}_t:t \ge \hat{\tau}_\mathrm{leaf})$ evolve independently of~$(\hat{\Phi}_t)$, with the law of a herds process started from the state~$\hat{\Xi}_{\hat{\tau}_\mathrm{leaf}}$ at time~$\hat{\tau}_\mathrm{leaf}$. By checking the jump rates of both Markov chains, we then see that both marginal processes indeed have the desired distributions.

	Now, fix~$k > 0$. By Lemma~\ref{lem_for_herds}, we can choose~$t_* > 0$ such that~$\mathbb{P}(|\hat{\mathcal{J}}^\Xi_{t_*}| > k) > 3\rho/4$. Next, by Lemma~\ref{lem_leaf}, we can choose~$h_1$ large enough that, if~$h\ge h_1$, then~$\mathbb{P}(\hat{\tau}_\mathrm{leaf} < t_*) < \rho/4$. We then have
	\[\mathbb{P}(|\hat{\mathcal{J}}^\Phi_{t_*}| > k) \ge \mathbb{P}(|\hat{\mathcal{J}}^{\Xi}_{t_*}|>k) - \mathbb{P}(\hat{\mathcal{J}}^\Xi_{t_*} \neq \hat{\mathcal{J}}^\Phi_{t_*})\ge \frac{3\rho}{4} - \mathbb{P}(\hat{\tau}_\mathrm{leaf} \le t_*) \ge \frac{\rho}{2}.\]
	
\end{proof}

\subsection{Irreducibility bound: proof of Lemma~\ref{lem_regardless_of}} \label{ss_proof_regardless}
\begin{proof}[Proof of Lemma~\ref{lem_regardless_of}]
	Let us say that an~$h$-herd~$(A,\alpha)$ is \textit{unitary} if~$|\alpha| = 1$. Note that a unitary~$h$-herd does not have active edges, so it does not split.  Let us say that a unitary~$h$-herd~$(A,\alpha) = (A,\{u\})$ is of~\textit{leaf type}  if~$\deg_A(u) = 1$; otherwise we have~$\deg_A(u) = d$ (by~\eqref{eq_deg_one_or_d}), in which case we say that~$(A,\alpha)$ is of \textit{interior type}. Using these definitions, we will prove the lemma in a few steps.\\

\noindent	\textit{Step 1: from~$h$-herd to unitary~$h$-herd}. We first claim that there exists~$p_0'>0$ (depending only on~$\mathsf{v},\lambda$) such that if~$(\Phi_t)$ is started from any non-empty state we have
	\[\mathbb{P}(\text{there is a unitary~$h$-herd at time~$t = 1$}) > p_0'.\]
	This follows from two observations. First, if a unitary~$h$-herd is formed at some time~$s \in (0,1]$ (or if it is already present at time zero), then it remains present until time one with a probability that is bounded away from zero (by not being involved in any birth or death events). Second, for any~$s \ge 0$ for which~$\Phi_s$ is still alive, any existing~$h$-herd that is not unitary contains at least one active edge whose splitting would produce at least one  new unitary~$h$-herd; such a splitting occurs with rate~$\mathsf{v}$.\\
	
\noindent \textit{Step 2: from unitary~$h$-herd to unitary~$h$-herd of interior  type}.	 We now claim that there exists~$p_0''>0$ (depending only on~$\mathsf{v},\lambda$) such that, if~$(\Phi_t)$ has at least one unitary~$h$-herd at time zero, then with probability at least~$p_0''$ we have
	\[\mathbb{P}(\text{there is a unitary~$h$-herd of interior type at time~$t = 1$}) > p_0''.\]
	To see this, first observe again that a unitary~$h$-herd of interior type that is formed at some time~$s \in (0,1]$ (or that is already present at time zero) stays unchanged until time one with a probability that is bounded away from zero.  Secondly, a unitary~$h$-herd of leaf type~$(A,\alpha) = (A,\{u\})$ can produce a unitary~$h$-herd of interior type by following the steps:~(i) the particle at~$u$ gives birth at its neighbouring position~$v$, so that the edge~$\{u,v\}$ becomes active; (ii) this active edge splits, forming the two~$h$-herds
	\[\mathcal{T}((A,\{u,v\}),\{u,v\},u),\quad \mathcal{T}((A,\{u,v\}),\{u,v\},v),\]
	the second of which is a unitary $h$-herd of interior type. The probability that these steps occur within one time unit of the dynamics is again bounded away from zero, uniformly in~$h$.\\

	\noindent \textit{Step 3: from unitary~$h$-herd of interior type to~$(\mathcal{B}_{\mathbb{T}}(u,h),\{u\})$}. Finally, assume that~$(\Phi_t)$ includes at time zero a unitary~$h$-herd of interior type,~$(A,\alpha) = (A, \{u\})$, and let~$v$ be a neighbor of~$u$. Assume that in one time unit of the dynamics, the following (and nothing else) occurs involving this herd:~(i) the particle at~$u$ gives birth at~$v$, so that the edge~$\{u,v\}$ becomes active; (ii) this active edge splits; (iii) the newly formed~$h$-herd~$(\mathcal{T}((A,\{u\}),\{u,v\},u)$ has no further updates until time one.
	The probability that all this occurs is larger than some~$p_0''' > 0$ which again does not depend on~$h$. Combining this with Lemma~\ref{lem_scheme}, we obtain that with probability at least~$(p_0''')^d$, at time~$t = d$ an~$h$-herd~$(\mathcal{B}_\mathbb{T}(u,h),\{u\})$ is present in the process.\\

	The statement of the lemma now follows with~$t_0 = d+2$ and~$p_0 = p_0' \cdot p_0'' \cdot (p_0''')^d$.
\end{proof}
\subsection{Eigenfunction of~$h$-herds}\label{ss_eigen}
For any subtree~$A$ of~$\mathbb{T}$, we denote by~$[A]$ the collection of all subtrees~$\tilde{A}$ of~$\mathbb{T}$ such that there is a graph isomorphism~$\varphi: A \to \tilde{A}$. Moreover, if~$\alpha$ is a set of vertices of~$A$, we denote by~$[(A,\alpha)]$ the set of pairs~$(\tilde{A},\tilde{\alpha})$, where~$\tilde{A}$ is a subtree of~$\mathbb{T}$,~$\tilde{\alpha}$ is a set of vertices of~$\tilde{A}$ and there is an isomorphism~$\varphi: A \to \tilde{A}$ with~$\varphi(\alpha) = \tilde{\alpha}$.

We define
\begin{equation*}
	[\mathscr{A}_h] := \{[A]: A \in \mathscr{A}_h\}, \qquad
	[\mathscr{A}_h']:= \{[(A,\alpha)]: (A,\alpha) \in \mathscr{A}_h'\}.
\end{equation*}

Given the~$h$-herds process~$(\Phi_t )_{t \ge 0}= (\mathcal{J}_t,\{(A^i_t,\alpha^i_t): i \in \mathcal{J}_t\})_{t \ge 0}$, we write
\[[\Phi_t] := (\mathcal{J}_t,\{[(A^i_t,\alpha^i_t)]: i \in \mathcal{J}_t\}),\quad t \ge 0.\]
Then,~$([\Phi_t])_{t \ge 0}$ is a continuous-time, multi-type branching process with (finite) space of types equal to~$[\mathscr{A}_h']$ (to be more precise, such a branching process is obtained from~$([\Phi_t])_{t \ge 0}$ by ignoring the indices, and only keeping track of the number of $h$-herds of each possible form).
 See Chapter~V.7 of~\cite{athreya} for a treatment of continuous-time, multi-type branching processes. We note  that~$([\Phi_t])_{t \ge 0}$ is irreducible in the sense that any of the types can produce any of the other types in a finite number of steps, as a consequence of Lemma~\ref{lem_regardless_of}.

By Perron-Frobenius theory, there exists a Perron-Frobenius eigenvalue~$\mu\in \mathbb{R}$ and a corresponding eigenvector~$f:[\mathscr{A}_h']\to \mathbb{R}$  such that, for any~$h$-herd~$(A,\alpha)$, the process~$(\Phi_t)$ started from a single~$h$-herd~$(A,\alpha)$ satisfies
\begin{equation}\label{eq_first_pf}
	\frac{\mathrm{d}}{\mathrm{d}t}\left. \mathbb{E} \left[ \sum_{i \in \mathcal{J}_t} f([(A^i_t,\alpha^i_t)])  \right] \right|_{t= 0+} = \mu \cdot f([(A,\alpha)]).
\end{equation}
Moreover, we have that~$\mu > 0$ if the process survives.

We abuse notation and denote the mapping~$(A,\alpha) \mapsto f([(A,\alpha)])$ again by~$f$. With this notation,~$f$ is a real-valued function of~$h$-herds that is invariant under the~$[\cdot]$ equivalence relation. Expressing the left-hand side of~\eqref{eq_first_pf} as a sum over all possible jumps in the dynamics, we obtain
\begin{equation}\label{eq_perron_frobenius}
	S_{\mathrm{death}}(A,\alpha) + S_{\mathrm{birth}}(A,\alpha) + S_{\mathrm{switch}}(A,\alpha) = \mu\cdot f(A,\alpha),
\end{equation}
where
\begin{align*}
	&S_{\mathrm{death}}(A,\alpha) := \sum_{u \in \alpha} \left( f(A,\alpha \backslash \{u\}) - f(A,\alpha) \right),\\[.2cm]
	&S_\mathrm{birth}(A,\alpha):= \lambda \cdot \sum_{u \in A \backslash \alpha} |\{v \in \alpha: v \sim u\}|\cdot \left( f(A,\alpha \cup \{u\}) - f(A,\alpha)\right),\\[.2cm]
	&S_\mathrm{split}(A,\alpha) := \mathsf{v} \cdot \sum_{\substack{e = \{u,v\}\\ \text{active edge}\\\text{in } (A,\alpha)}} \left( f(\mathcal{T}_h((A,\alpha),e,u)) + f(\mathcal{T}_h((A,\alpha),e,v)) - f(A,\alpha)\right).
\end{align*}

\section{Switching graph and embedded~$h$-herds}\label{s_embed}

We now return to the contact process on the switching random~$d$-regular graph, aiming at proving Theorem~\ref{thm_main} with the aid of the auxiliary processes studied in the previous sections. In Section~\ref{sec_prelim}, we go over the joint construction of the evolving graph and the contact process, repeating some of the definitions that were given in the Introduction with some more details, and also explaining how a graphical construction of the particle system is implemented in this  setting. In Section~\ref{ss_splitting_trees}, we study embeddings inside a random~$d$-regular graph of trees and~$h$-herds from the collections~$\mathscr{A}_h$ and~$\mathscr{A}_h'$ from Section~\ref{s_hherds}. We also explain how a switching of a pair of edges of the~$d$-regular graph may induce a splitting of such an embedded tree or~$h$-herd. Next, in Section~\ref{ss_def_of_emb} we will study a family of embedded~$h$-herds evolving in~$(G_t)$, which we call the \textit{embedded~$h$-herds process}. Finally, in Section~\ref{ss_martingale} we develop some martingale estimates for the embedded~$h$-herds process and give the proof of Theorem~\ref{thm_main}.

\subsection{Preliminaries on switching graph and contact process}\label{sec_prelim}

Let~$n \in \mathbb{N}$. We let~$V = [n]$, and fix a perfect matching~$\varphi: V \times [d] \to V \times [d]$ (that is,~$\varphi$ is a bijection with no fixed point and equal to its own inverse). Elements of~$V$ are vertices, and elements of~$V\times [d]$ are half-edges; we call~$(x,i)\in V \times [d]$ the~$i$-th half-edge of vertex~$x$. A pair of the form~$\{(x,a),(x',a')\}$ with~$(x',a') =\varphi((x,a))$ is called an edge between~$x$ and~$x'$. Letting~$E$ denote the set of edges, we obtain  a multi-graph~$G=(V,E)$. If an edge exists between~$x$ and~$x'$, we say that these vertices are neighbors and denote this relation by~$x \sim x'$.  The degree of any vertex is defined as the number of half-edges that it possesses, and is thus equal to~$d$ for any vertex (this may differ from the number of neighbors of the vertex). Finally, we let~$\mathcal{G}_n$ denote the set of all graphs with~$n$ vertices (and degrees~$d$) obtained in this way, as~$\varphi$ ranges over all perfect matchings on the set of half-edges.

Given~$G \in \mathcal{G}_n$, we define the set of \textit{marks} on~$G$ as the set of pairs of the form~$\mathsf{m}=(\{e,e'\},\sigma)$, where~$e,e'$ are distinct edges of~$G$ and~$\sigma \in \{+,-\}$ (a mark with~$\sigma = +$ is called positive, and a mark with~$\sigma = -$ is negative). We define the graph~$\Gamma^\mathsf{m}(G)$ as follows: letting~$e = \{(x,a),(y,b)\}$ and~$e' = \{(x',a'),(y',b')\}$ with~$(x,a) < (y,b)$ and~$(x',a')< (y',b')$ in the lexicographic order on half-edges, we set~$\Gamma^\mathsf{m}(G)$ as the graph equal to~$G$, except that
\[e,e' \text{ are replaced by } \begin{cases} \{(x,a),(x',a')\},\;\{(y,b),(y',b')\} &\text{if } \sigma = +;\\\{(x,a),(y',b')\},\;\{(y,b),(x',a')\}&\text{if } \sigma = -\end{cases}
	\]
	(or more formally, the perfect matching of half-edges that produces~$G$ is replaced by the one that produces this new set of edges).

	Fix~$\mathsf{v} \ge 0$ and let~$\upupsilon_n := \frac{\mathsf{v}}{nd}$. We now define a continuous-time Markov chain on~$\mathcal{G}_n$ by prescribing that for each~$G$ and for each mark~$\mathsf{m}$ of~$G$, the chain jumps from~$G$ to~$\Gamma^\mathsf{m}(G)$ with rate~$\upupsilon_n$. We assume that the initial state of the chain is uniformly distributed on~$\mathcal{G}_n$. This gives rise to a dynamic random graph~$(G_t)_{t \ge 0}$, the \textit{switching random~$d$-regular graph on~$n$ vertices with switch rate~$\upupsilon_n$}. The jump mechanism of the chain is reversible with respect to the uniform distribution, so the dynamics of~$(G_t)_{t \ge 0}$ is stationary. 

Next, we formally define the joint evolution~$(G_t,\xi_t)_{t \ge 0}$ of the switching random~$d$-regular graph and the contact  process. This is the Markov chain with state space~$\mathcal{G}_n \times \{0,1\}^{V}$ with generator given by
\begin{align*}
	Lf(G,\xi):= &\upupsilon_n \sum_{\mathsf{m}} (f(\Gamma^\mathsf{m}(G),\xi)-f(G,\xi)) \\
	&+ \sum_{x \in \xi} (f(G,\xi\backslash \{x\})-f(G,\xi))+ \lambda\sum_{x \notin \xi} \sum_{y \in \xi} \mathcal{N}_G(x,y)\cdot (f(G,\xi \cup \{x\}) - f(G,\xi)),
\end{align*}
where we adopt the abuse of notation of associating~$\xi \in\{0,1\}^V$ with the set~$\{x:\xi(x) = 1\}$, and~$\mathcal{N}_G(x,y)$ denotes the number of edges in~$G$ between~$x$ and~$y$. Note that the first summand above makes it so that~$(G_t)_{t \ge 0}$ follows the dynamics described earlier, whereas the second and third summands give the death and birth mechanisms of the contact process, respectively.

For coupling purposes, it is useful to notice that the process~$(G_t,\xi_t)$ can be obtained from~$(G_t)$ combined with a standard Poisson graphical construction; let us briefly explain this. Together with the process~$(G_t)$, we take a family of independent Poisson processes~$\{(R^x_t)_{t \ge 0}: x \in V\}$, each with rate one, and a second family of independent Poisson processes~$\{(T^{(x,a)}_t)_{t \ge 0}:(x,a) \in V \times [d]\}$, each with rate~$\lambda$. Now assume that~$\xi_0 \in \{0,1\}^V$ is given. Let~$0 < t_1 < t_2 < \cdots$ denote the arrival times  of all Poisson processes mentioned above, in increasing order. We let~$\xi_t$ be constant in the intervals~$(0,t_1),(t_1,t_2),\ldots$, and otherwise define it recursively as follows. Say that~$\xi_{t_k-}$ is already defined. First assume that~$t_k \in R^x$ for some~$x \in V$. We then set~$\xi_{t_k} = \xi_{t_k-} \backslash \{x\}$ (we say that a recovery occurred at~$x$). Next, assume that~$t_k \in T^{(x,a)}$ for some half-edge~$(x,a)$, and let~$y$ be the vertex owning the half-edge to which~$(x,a)$ is matched in~$G_{t_k-}$ (or in~$G_{t_k}$, since with probability one the two graphs are equal). We then let~$\xi_{t_k} = \xi_{t_k-} \cup \{y\}$ in case~$x \in \xi_{t_k-}$ (we say that~$x$ transmits the infection to~$y$); otherwise we let~$\xi_{t_k} = \xi_{t_k-}$.

This construction yields a monotonicity property that is well known for the classical contact process. Let us explain how it is formulated in the present context. Assume that the processes~$(G_t)$,~$\{(R^x_t)\}$ and~$\{(T^{(x,a)}_t)\}$ are all given, and assume that~$\xi_0,\xi_0' \in \{0,1\}^V$ satisfy~$\xi_0 \le \xi_0'$ in the partial order of~$\{0,1\}^V$. Then, using these processes for the graphical construction of both~$(G_t,\xi_t)$ and~$(G_t,\xi'_t)$, where~$(\xi_t)$ is started from~$\xi_0$ and~$(\xi'_t)$ is started from~$\xi'_0$, we obtain~$\xi_t \le \xi'_t$ for all~$t \ge 0$.

We will be mostly interested in considering~$(G_t,\xi_t)$ for the contact process started from~$\xi_0 \equiv 1$. Assuming this is the case, we define
\[\tau_{(G_t)} := \inf\{t \ge 0: \xi_t \equiv 0\},\]
the \textit{extinction time} of the infection. This is the stopping time that appears in the statement of Theorem~\ref{thm_main}.

We conclude this section with a result concerning the number of loops in the switching random regular graph. Given~$G \in \mathcal{G}_n$ and~$m \in \mathbb{N}$,~$m \ge 2$, a \textit{loop of length $m$} in~$G$ is a set of~$m$ edges of~$G$ that admits an enumeration~$e_1=\{(x_1,a_1),(y_1,b_1)\}$,~$\ldots$,~$e_{m} = \{(x_{m},a_{m}),(y_{m},b_{m})\}$ so that:~$x_1,\ldots, x_{m-1}$ are all distinct,~$y_i = x_{i+1}$ for~$1 \le i \le m-1$ and~$y_{m} = x_1$. A loop of length one is simply a self-loop, that is, an edge whose half-edges both belong to the same vertex.

Let~$\ell_m(G)$ denote the number of loops of length at most~$m$ in~$G$.
\begin{proposition}\label{prop_no_loops} For all~$m\in \mathbb{N}$ and~$\kappa> 0$ there exists~$c=c(\mathsf{v},h,\kappa) > 0$ such that
	\[
		\mathbb{P}(\ell_m(G_t) > \kappa \cdot n \text{ for some } t \le \exp(c\cdot n)) < \exp(-c \cdot n)\quad \text{for all } n \in \mathbb{N}.
	\]
\end{proposition}

\begin{proof}
	Fix~$m \in \mathbb{N}$ and~$\kappa > 0$. 	We start with a simple observation. Since there are~$\frac{nd}{2}\cdot \left(\frac{nd}{2}-1\right)$ (unordered) pairs of edges in~$G$ and each pair of edge can be involved in a switch in two different ways, the total rate at which switches occur in~$(G_t)$ is
	\[r := \upupsilon_n\cdot \frac{nd}{2}\cdot \left(\frac{nd}{2}-1\right) = \frac{\mathsf{v}}{2}\cdot \left(\frac{nd}{2}-1\right).\]
	Hence, for any~$t > 0$ we have that
	\[\mathbb{P}(G_s = G_t \text{ for all } s \in [t,t+r^{-1}]) = e^{-1}.\]

	Let~$\sigma_\ell := \inf\{t \ge 0:\;\ell_m(G_t) > \kappa \cdot n\}$. Fix~$\bar{t} > 0$ and let~$A_\ell$ be the event that~$\sigma_\ell \le \bar{t}$ and the graph is unchanged in the time interval~$[\sigma_\ell,\sigma_\ell + r^{-1}]$; by the strong Markov property we have~$\mathbb{P}(A_\ell) = \mathbb{P}(\sigma_\ell \le \bar{t}) \cdot e^{-1}$. Next, we have
	\[r^{-1} \cdot \mathds{1}_{A_\ell} \le \int_0^{\bar{t}+r^{-1}} \mathds{1}\{\ell_m(G_s) > \kappa \cdot n\}\;\mathrm{d}s;\]
	integrating, multiplying by~$e$ and using stationarity gives
	\[  \mathbb{P}(\sigma_\ell \le \bar{t}) =e \cdot  \mathbb{P}(A_\ell) \le e\cdot r \cdot  \int_0^{\bar{t}+r^{-1}} \mathbb{P}(\ell_m(G_s) > \kappa \cdot n)\;  \mathrm{d}s = e\cdot r \cdot (\bar{t} + r^{-1}) \cdot \mathbb{P}(\ell_m(G_0) > \kappa \cdot n). \]
	By Theorem 2.19 in \cite{wor13} and stationarity, there exists~$c_0 = c_0(h,\kappa) > 0$ such that
	\[\mathbb{P}(\ell_m(G_0) > \kappa \cdot n) < \exp(-c_0\cdot n) \quad \text{ for all } n \in \mathbb{N}.\]
	The desired bound now follows by taking~$\bar{t} = \exp(\tfrac{c_0}{2}\cdot n)$.
\end{proof}

\subsection{Splitting trees in $d$-regular graph}\label{ss_splitting_trees}
Let~$G \in \mathcal{G}_n$. An \textit{embedded~$h$-herd in~$G$} is a pair of the form~$(B,\beta)$, where~$B$ is a subgraph of~$G$ that is isomorphic to some tree~$A \in \mathscr{A}_h$, and~$\beta$ is a subset of the set of vertices of~$B$. In these circumstances, let~$\varphi:A \to B$ be an isomorphism and~$\alpha = \varphi^{-1}(\beta)$; we then abuse notation and (recalling the notation from Section~\ref{ss_eigen}) write
\begin{align}\nonumber&f(B,\beta)=f(A,\alpha),\\
	\label{eq_death_B}&S_{\mathrm{death}}(B,\beta) = \sum_{x \in \beta} \left(f(B,\beta \backslash \{x\}) - f(B,\beta)\right)= S_{\mathrm{death}}(A,\alpha) ,\\
	\label{eq_birth_B}&S_{\mathrm{birth}}(B,\beta) =\lambda \cdot \sum_{x \in B\backslash \beta} |\{y \in \beta: y \stackrel{B}{\sim} x \}|\cdot \left(f(B,\beta \cup \{x\}) - f(B,\beta)\right)= S_{\mathrm{birth}}(A,\alpha).
\end{align}
Moreover, we define the set of active edges of~$(B,\beta)$ as the set of  edges of~$B$ whose removal would break~$B$ into two components, both intersecting~$\beta$. We will often omit the word `embedded' when it is clear from the context, so we will simply refer to~$(B,\beta)$ as an~$h$-herd.

We will now define a splitting operation on~$(B,\beta)$ which is analogous to the splitting of~$(A,\alpha)$ into $\mathcal{T}_h((A,\alpha),e,u)$ and~$\mathcal{T}_h((A,\alpha),e,v)$, where~$e=\{u,v\}$ is some active edge of~$(A,\alpha)$.
Although this splitting operation is somewhat clumsy to describe formally, it is very simple, and can be readily understood with the aid of Figure~\ref{fig:embed_two}.\\

\noindent \textbf{Splitting of~$(B,\beta)$ through active edge.}
We start by fixing an edge~$e = \{(x^1,a^1),(x^2,a^2)\}$ of~$B$ that is active with respect to~$\beta$. We will also need some ``extra space'' inside~$G$ where the augmentations that follow the breaking of~$e$ can be performed. For this purpose, we fix an edge~$e'=\{(y^1,b^1),(y^2,b^2)\}$ of~$G$ with the property that
\begin{equation}\label{eq_assumption_e}\mathcal{B}_G(y^1,2h),\mathcal{B}_G(y^2,2h)\text{ are both trees and both disjoint from~$B$}.\end{equation} Let us now define the following auxiliary graphs:
\begin{itemize}
	\item let~${D}^1,{D}^2$ be the two connected components of~$B$ that remain after the removal of~$e$, with~${D}^1$ containing~$x^1$ and~${D}^2$ containing~$x^2$ -- with similar notation as in Section~\ref{ss_splitting_trees}, we have
		\[D^1 = \mathcal{L}(B,e,x^1) \quad \text{and}\quad D^2 = \mathcal{L}(B,e,x^2);\]
	\item for~$1 \le k \le h$, let~$\tilde{D}^{1,k}$ denote the subgraph of~$\mathcal{B}_G(y^1,h)$ that is induced by the set of vertices that can be reached from~$y^1$ by a path (in~$\mathcal{B}_G(y^1,h)$) of length at most~$k-1$ that does not contain~$y^2$;
	\item similarly, for~$1 \le k \le h$, let~$\tilde{D}^{2,k}$ denote the subgraph of~$\mathcal{B}_G(y^2,h)$ that is induced by the set of vertices that can be reached from~$y^2$ by a path (in~$\mathcal{B}_G(y^2,h)$) of length at most~$k-1$ that does not contain~$y^1$.
\end{itemize}

\begin{figure}[htb]
	\begin{center}
		\setlength\fboxsep{0pt}
		\setlength\fboxrule{0.4pt}
		\fbox{\includegraphics[width = \textwidth]{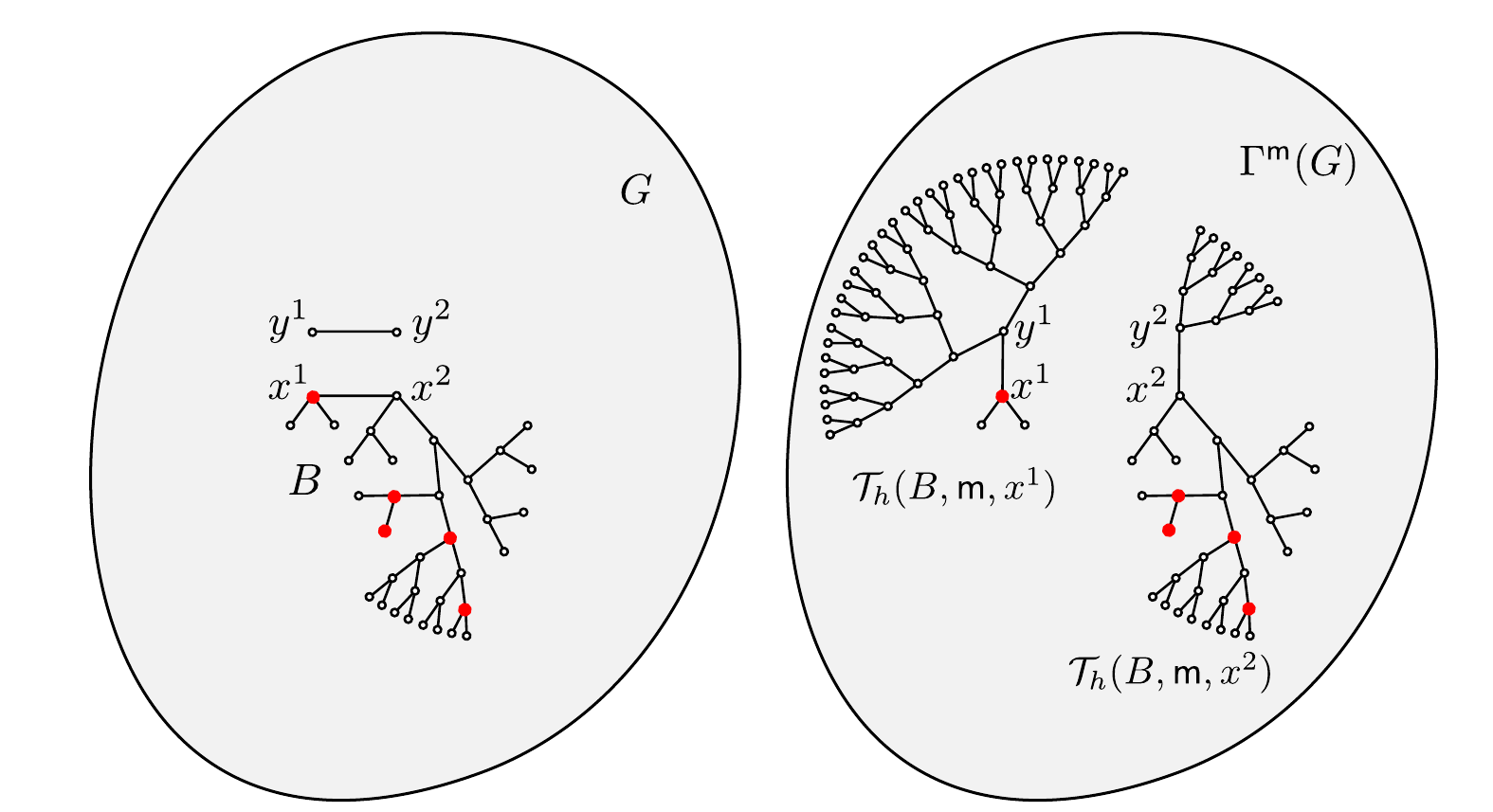}}
	\end{center}
\vspace*{-.5cm}
	\caption{Illustration of splitting involving an active edge, with~$h=5$.}
	\label{fig:embed_two}
\end{figure}

Now fix~$\sigma\in\{+,-\}$ and set~$\mathsf{m}=(\{e,e'\},\sigma)$; without loss of generality, assume that~$\sigma$ is so that~$\mathsf{m}$ associates~$x^1$ with~$y^1$ and~$x^2$ with~$y^2$. Recall that~$\Gamma^{\mathsf{m}}(G)$ denotes the graph obtained from~$G$ after performing the switch encoded by~$\mathsf{m}$.  We will now define~$\mathcal{T}_h(B,\mathsf{m},x^1)$ and~$\mathcal{T}_h(B,\mathsf{m},x^2)$, both of which will be subgraphs of~$\Gamma^{\mathsf{m}}(G)$. 
Construct~$\mathcal{T}_h(B,\mathsf{m},x^1)$ by including in it:~$D^1$, the (new) edge~$\{(x^1,a^1),(y^1,b^1)\}$, and~$\tilde{D}^{1,k_*}$, where~$k_*$ is the largest value of~$k$ so that the resulting graph has diameter at most~$2h$. Similarly, construct~$\mathcal{T}_h(B,\mathsf{m},x^2)$ by including in it:~$D^2$, the (new) edge~$\{(x^2,a^2),(y^2,b^2)\}$, and~$\tilde{D}^{2,k_{**}}$, where~$k_{**}$ is the largest value of~$k$ such that  the resulting graph has diameter at most~$2h$. Finally, define
	\begin{align*} &\mathcal{T}_h((B,\beta),\mathsf{m},x^1):= (\mathcal{T}_h(B,\mathsf{m},x^1),\beta \cap D^1), \\
	&\mathcal{T}_h((B,\beta),\mathsf{m},x^2):= (\mathcal{T}_h(B,\mathsf{m},x^2),\beta \cap D^2).\end{align*}
Note that both of these are embedded~$h$-herds in~$\Gamma^\mathsf{m}(G)$.
 
 Again take~$(A,\alpha) \in \mathscr{A}_h'$ that is mapped under some isomorphism~$\varphi$ to~$(B,\beta)$, let~$u^1 = \varphi^{-1}(x^1)$,~$u^2 = \varphi^{-1}(x^2)$, and~$e_0=\{u^1,u^2\}$. We have that there exists an  isomorphism that maps the~$h$-herd~$\mathcal{T}_h((A,\alpha),e_0,u^1)$ into~$\mathcal{T}_h((B,\beta),\mathsf{m},x^1)$, and similarly there exists an isomorphism that maps~$\mathcal{T}_h((A,\alpha),e_0,u^2)$ into~$\mathcal{T}_h((B,\beta),\mathsf{m},x^2)$.  In particular, we have
\begin{equation}\label{eq_relation_ABm}\begin{split}
	&f(\mathcal{T}_h((B,\beta),\mathsf{m},x^1)) = f(\mathcal{T}_h((A,\alpha),e_0,u^1)),\\
&f(\mathcal{T}_h((B,\beta),\mathsf{m},x^2)) = f(\mathcal{T}_h((A,\alpha),e_0,u^2).\end{split}
\end{equation}

Still assuming that~$e'$ satisfies~\eqref{eq_assumption_e}, now define
\begin{equation}\label{eq_def_S_switch}\begin{split}& S_{\mathrm{split}}(B,\beta,e')\\
&\quad:=\upupsilon_n\cdot \sum_{\substack{\mathsf{m}=(\{e,e'\},\sigma):\\e=\{x^1,x^2\}\\ \text{ active in } (B,\beta),\\\sigma \in \{+,-\}}} \left( f(\mathcal{T}_h((B,\beta),\mathsf{m},x^1)) + f(\mathcal{T}_h((B,\beta),\mathsf{m},x^2)) - f(B,\beta)  \right).\end{split}\end{equation}
Using~\eqref{eq_relation_ABm}, we obtain that
\begin{equation}\label{eq_relation_SBSA}\begin{split}
	S_{\mathrm{split}}(B,\beta,e') &= \upupsilon_n \cdot 2 \cdot \sum_{\substack{e=\{u^1,u^2\} \\\text{active in }(A,\alpha)}} \left( f(\mathcal{T}_h((A,\alpha),e,u^1)) + f(\mathcal{T}_h((A,\alpha),e,u^2)- f(A,\alpha)\right)\\&= \frac{2\upupsilon_n}{\mathsf{v}}\cdot S_{\mathrm{split}}(A,\alpha).
\end{split}\end{equation}

\begin{lemma}
	\label{lem_single_herd}
	Assume that~$\lambda$ is larger than $\bar{\lambda}(\mathsf{v},h)$, the critical value for the~$h$-herds process~$(\Phi_t)_{t \ge 0}$, and let~$\mu$ and~$f$ be the associated Perron-Frobenius eigenvalue and eigenfunction, respectively, as in Section~\ref{ss_eigen}.
	 There exists~$\varepsilon_0= \varepsilon_0(\lambda,\mathsf{v},h) > 0$ such that the following holds for any~$n$. Assume that~$G\in \mathcal{G}_n$ and~$(B,\beta)$ is an embedded~$h$-herd in~$G$. Let~$\Lambda$ be a set of edges of~$G$ with the properties that every edge $e'=\{(y^1,b^1),(y^2,b^2)\}\in\Lambda$ satisfies \eqref{eq_assumption_e} with respect to~$B$, and~$|\Lambda| > (1-\varepsilon_0)\frac{dn}{2}$.
	 We then have
	\begin{equation}\label{eq_sum_want}
		S_{\mathrm{death}}(B,\beta) + S_{\mathrm{birth}}(B,\beta) + \sum_{e' \in \Lambda} S_{\mathrm{split}}(B,\beta, e')\ge \frac{\mu}{2}\cdot f(B,\beta).
	\end{equation}
\end{lemma}

\begin{proof} Let~$\varepsilon_0 > 0$ be a small constant to be chosen later, and let~$G$,~$(B,\beta)$ and~$\Lambda$ be as in the statement.
	Fix~$A \in \mathscr{A}$ such that~$A$ is isomorphic to~$B$, and let~$\alpha\subseteq A$ be the set of vertices corresponding to~$\beta$ under the isomorphism.
Using~\eqref{eq_death_B},~\eqref{eq_birth_B} and~\eqref{eq_relation_SBSA}, we have that the left-hand side of~\eqref{eq_sum_want} equals
	\begin{align}\nonumber&S_{\mathrm{death}}(A,\alpha) + S_{\mathrm{birth}}(A,\alpha) + \frac{2\upupsilon_n|\Lambda|}{\mathsf{v}}\cdot S_{\mathrm{split}}(A,\alpha)\\
		&\nonumber=S_{\mathrm{death}}(A,\alpha) + S_{\mathrm{birth}}(A,\alpha) +S_{\mathrm{split}}(A,\alpha) + \left( \frac{2\upupsilon_n|\Lambda|}{\mathsf{v}}-1\right) \cdot S_{\mathrm{split}}(A,\alpha)\\
		&\nonumber\ge S_{\mathrm{death}}(A,\alpha) + S_{\mathrm{birth}}(A,\alpha) +S_{\mathrm{split}}(A,\alpha) -\left(1-\frac{|\Lambda|}{dn/2} \right)\cdot |S_{\mathrm{split}}(A,\alpha)| \\
	\label{eq_my_mu_bound}&\ge \mu\cdot f(A,\alpha) - \varepsilon_0 \cdot |S_{\mathrm{split}}(A,\alpha)|,\end{align}
	where in the first inequality we used that~$\upupsilon_n = \frac{\mathsf{v}}{dn}$ and~$|\Lambda| \le \frac{dn}{2}$ (the number of edges of~$G$) and in the second inequality we used~\eqref{eq_perron_frobenius} and~$|\Lambda| > (1-\varepsilon_0) \cdot \frac{dn}{2}$.

	Now, using the  definition of~$S_{\mathrm{split}}(A,\alpha)$, it is  easy to check that
	\[|S_\mathrm{split}(A,\alpha)| \le 3\mathsf{v}\cdot |\mathcal{B}_\mathbb{T}(o,2h)| \cdot f_{\mathrm{max}},\] where~$f_{\mathrm{max}}$ is the largest value attained by the function~$f$. Hence, by taking~$\varepsilon_0 \le \frac{\mu \cdot f_\mathrm{min}}{6\mathsf{v}\cdot |\mathcal{B}_\mathbb{T}(o,2h)|\cdot f_\mathrm{max}}$, where~$f_\mathrm{min}$ is the smallest value attained by~$f$, we obtain that the expression in~\eqref{eq_my_mu_bound} is larger than
	\[{\mu}\cdot f(A,\alpha) - 3\mathsf{v}\cdot |\mathcal{B}_\mathbb{T}(o,2h)|\cdot f_\mathrm{max} \cdot \frac{\mu \cdot f_\mathrm{min}}{6 \mathsf{v}\cdot  |\mathcal{B}_\mathbb{T}(o,2h)|\cdot f_\mathrm{max}} \ge  \frac{\mu}{2}\cdot f(A,\alpha) = \frac{\mu}{2}\cdot f(B,\beta).\]
\end{proof}

We finally give one more definition, namely the splitting of~$(B,\beta)$ through an inactive edge. This operation has no corresponding mechanism in the $h$-herds process, but it will be important for the process embedded in the~$d$-regular graph to be defined shortly. The operation will only produce one new $h$-herd (as opposed to two new $h$-herds in the splitting through an active edge), since one of the components of~$B$ after the removal of~$e$ does not intersect~$\beta$ and will be discarded. See Figure~\ref{fig:embed_one}.  \\

\noindent  \textbf{Splitting of~$(B,\beta)$ through inactive edge}. Fix~$(B,\beta)$ as above. Let $e = \{(x^1,a^1),(x^2,a^2)\}$ be an inactive edge of~$B$ with respect to~$\beta$. As before, using the notation of Section~\ref{ss_splitting_trees}, let
\[D^1 = \mathcal{L}(B,e,x^1),\qquad D^2 = \mathcal{L}(B,e,x^2).\]
 Since~$e$ is inactive, we can also assume that~$\beta$ is contained in~$D^1$.  Let~$e' = \{(y^1,b^1),(y^2,b^2)\}$ be an edge of~$G$ satisfying~\eqref{eq_assumption_e} with respect to~$B$, and let~$\sigma \in \{+,-\}$ be such that the mark~$\mathsf{m} = (\{e,e'\},\sigma)$ matches~$x^1$ to~$y^1$ and~$x^2$ to~$y^2$. Let~$\tilde{D}$ denote the subgraph of~$\mathcal{B}_G(y^1,2h)$ that is induced by the set of vertices that can be reached from~$y^1$ by a path (in~$\mathcal{B}_G(y^1,2h)$) of length at most~$2h$ that does not contain~$y^2$.
Now, fix a graph isomorphism~$\psi$ from~$D^2$ to a subgraph of~$\tilde{D}$ with~$\psi(x^2) = y^1$ (note that this is possible because the diameter of~$D^2$ is at most the diameter of~$B$, so at most~$2h$). Finally, define~$\mathcal{T}'_h(B,\mathsf{m})$ as the subgraph of~$\Gamma^\mathsf{m}(G)$ constructed by including in it:~$D^1$, the (new) edge~$\{(x^1,a^1),(y^1,b^1)\}$, and the subgraph of~$\tilde{D}$ induced by~$\psi(D^2)$. Finally let
\[\mathcal{T}'_h((B,\beta),\mathsf{m}) :=(\mathcal{T}'_h(B,\mathsf{m}),\beta).\]

The idea of this operation is that~$\mathcal{T}'_h((B,\beta),\mathsf{m})$ is isomorphic to~$(B,\beta)$; in particular,
we have~$f(\mathcal{T}'_h((B,\beta),\mathsf{m})) = f(B,\beta)$.

\begin{figure}[htb]
	\begin{center}
		\setlength\fboxsep{0pt}
		\setlength\fboxrule{0.5pt}
		\fbox{\includegraphics[width = 0.8\textwidth]{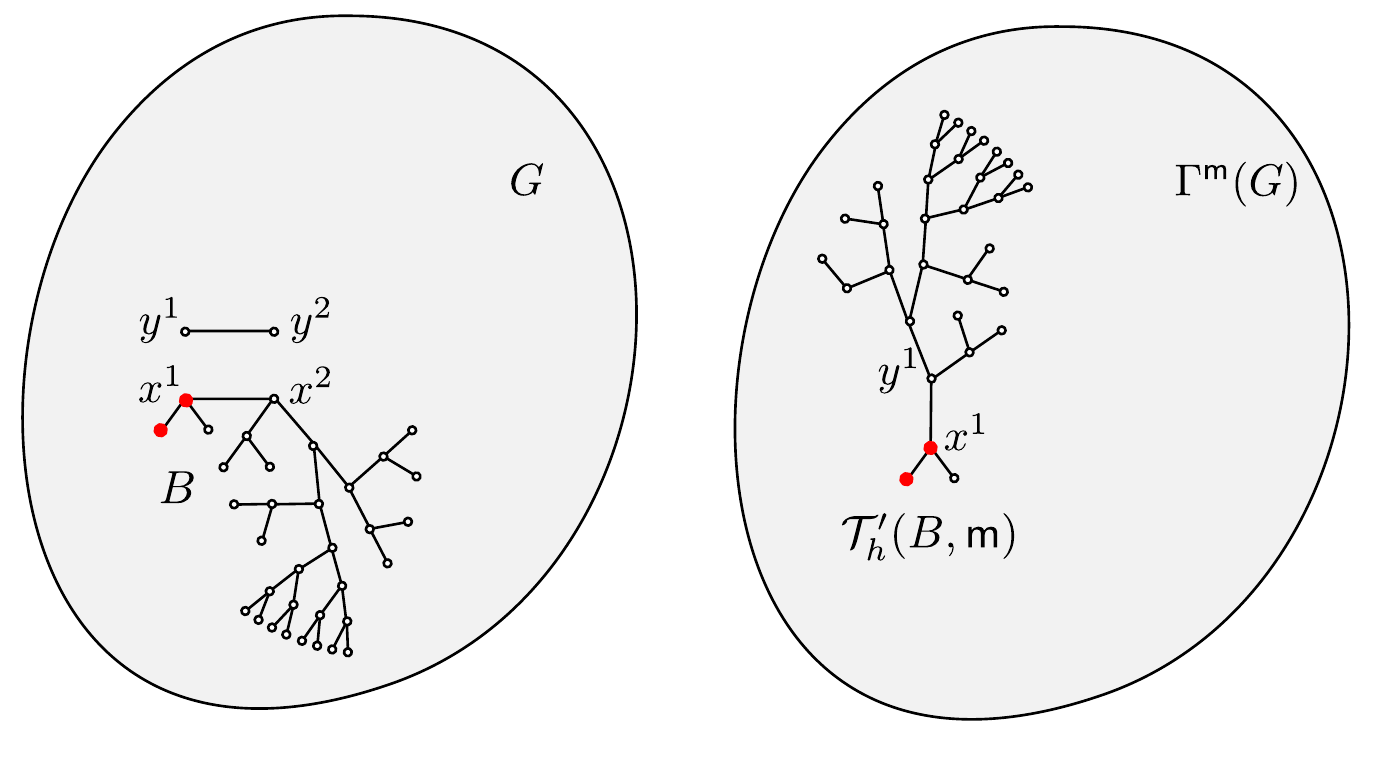}}
	\end{center}
\vspace*{-.5cm}
	\caption{Illustration of splitting involving an inactive edge, with~$h=5$.}
	\label{fig:embed_one}
\end{figure}

\subsection{Definition of the embedded~$h$-herds process}\label{ss_def_of_emb}
We will now define a process~$(G_t,\Psi_t)_{t \ge 0} = (G_t, (\mathcal{J}_t, \{(B^i_t,\beta^i_t): i \in \mathcal{J}_t\}))_{t \ge 0}$, which we describe informally as follows. The first marginal,~$(G_t)$, is simply a switching graph taking values in~$\mathcal{G}_n$ with switch rate~$\upupsilon_n=\frac{v}{nd}$. The second marginal,~$(\Psi_t)$, is given at each time~$t$ by a set of indices~$\mathcal{J}_t$ (which we again take as natural numbers), each of which refers to an embedded~$h$-herd~$(B^i_t,\beta^i_t)$ of~$G_t$. These~$h$-herds are such that~$B^i_t \cap B^j_t = \varnothing$ if~$i \neq j$.

Formally, the dynamics of~$(G_t,\Psi_t)$ is again described as a continuous-time Markov chain, with the dynamics given by jumps of three different types: contact births within~$h$-herds, contact deaths with removal of empty~$h$-herds,  and switches. Births and deaths occur in exactly the same  way as in the~$h$-herds process described in Section~\ref{ss_def_Phi}, so we refrain from repeating the description. Note in particular that births and deaths have no effect on the graph~$G_t$, but only on the indexed set of~$h$-herds~$\Psi_t$.

Let us now explain the effect of switches. Assume that the current state of the chain is \[(G,\Psi)=(G,(\mathcal{J}, \{(B^i,\beta^i):i \in \mathcal{J}\}).\] For each switch mark~$\mathsf{m} = (\{e,e'\},\sigma)$, where~$e,e'$ are edges of~$G$ and~$\sigma \in \{+,-\}$, the switch with mark~$\mathsf{m}$ occurs with rate~$\upupsilon_n=\frac{\mathsf{v}}{dn}$. The graph~$G$ is then replaced by the graph~$\Gamma^\mathsf{m}(G)$, as explained in Section~\ref{sec_prelim}. The indexed set of~$h$-herds~$\Psi$ is also replaced by an updated version~$\Gamma^\mathsf{m}(\Psi)$, which we will define below, according to different cases. First, define
\begin{equation}\label{eq_def_of_Lamb}\Lambda(G,\Psi):=\left\{ \begin{array}{l} \tilde{e}= \{(z^1,h^1),(z^2,h^2)\} \text{ edge of } G\backslash (\cup_{i \in \mathcal{J}} B^i) \text{ such that }\\[.1cm]
\mathcal{B}_G(z^1,2h)\text{ and } \mathcal{B}_G(z^2,2h) \text{ are both trees and both disjoint from }\cup_{i \in \mathcal{J}} B^i\end{array}\right\}.\end{equation}
Now (still fixing a mark~$\mathsf{m} = (\{e,e'\},\sigma)$ with~$e,e'$ edges of~$G$), we distinguish the cases: 
\begin{itemize}
	\item if neither~$e$ nor~$e'$ is contained in~$\cup_{i \in \mathcal{J}}B^i$, then the switch is called \textit{neutral} and we set~$\Gamma^\mathsf{m}(\Psi) = \Psi$;
	\item if one of the two edges (say,~$e$) is contained in~$\cup_{i \in \mathcal{J}}B^i$, and the other edge ($e'$) is contained in~$\Lambda(G,\Psi)$, then the switch is called \textit{good}. Further, letting~$i_* \in \mathcal{J}$ denote the index of the~$h$-herd in which~$e$ is contained,
		\begin{itemize}
			\item[-] we say that the switch is \textit{good and active} if~$e$ is an active edge of~$(B^{i_*},\beta^{i_*})$. In that case, we define~$\Gamma^\mathsf{m}(\Psi)$ by replacing the~$h$-herd with index~$i_*$ by the two~$h$-herds~$\mathcal{T}_h((B^{i_*},\beta^{i_*}),\mathsf{m},x^1)$ and $\mathcal{T}_h((B^{i_*},\beta^{i_*}),\mathsf{m},x^2)$, where~$x^1$ and~$x^2$ are the two vertices of~$e$; we leave all other~$h$-herds of~$\Psi$ unchanged (and update the index set accordingly);
			\item[-] we say that the  switch is \textit{good and inactive} if~$e$ is an inactive edge of~$(B^{i_*},\beta^{i_*})$. In that case, we define~$\Gamma^\mathsf{m}(\Psi)$ by replacing the~$h$-herd with index~$i_*$ by~${\mathcal{T}}_h'((B^{i_*},\beta^{i_*}),\mathsf{m})$; this new herd also receives the index~$i^*$, and the remaining~$h$-herds are left unchanged;
		\end{itemize}	
	\item in any other case (meaning: either if both~$e,e'$ are contained in~$\cup_{i\in\mathcal{J}}B^i$, or if one of them is contained in this union and the other is not, but also not in~$\Lambda(G,\Psi)$), the switch is called \textit{bad}. In that case,~$\Gamma^\mathsf{m}(\Psi)$ is given by deleting from~$\Psi$ the~$h$-herd (or~$h$-herds) that contain~$e$ or~$e'$ (or both); all other~$h$-herds are left unchanged, and the index set is updated accordingly.
\end{itemize}

Hence, the embedded~$h$-herds process~$(\Psi_t)$ evolves together with~$(G_t)$ essentially in the same way as the~$h$-herds  process~$(\Phi_t)$, with the only difference that some of the edge switches try to cause~$h$-herds  to expand towards occupied or undesirable regions of the graph, in which case the~$h$-herds involved in the accident are removed. Most of the effort of the remaining of this section will be to argue that, provided that~$G_t$ has few loops and~$\Psi_t$ does not occupy a large portion of~$G_t$, then accidents  are very unlikely, so the process tends to grow as~$\Phi_t$ would. This will be implemented by studying the growth of the process
\[X_t:= \sum_{i \in \Psi_t} f(B^i_t,\beta^i_t),\quad t \ge 0,\]
where~$f$ is the Perron-Frobenius eigenfunction defined in Section~\ref{ss_eigen}.

\subsection{Martingale estimates and survival}\label{ss_martingale}
For the rest of this section, we fix~$\mathsf{v}$,~$h$ and~$\lambda$ satisfying~$\lambda > \bar{\lambda}(\mathsf{v},h)$. We take the corresponding Perron-Frobenius eigenvalue~$\mu$ and eigenfunction~$f$, and let~$f_{\mathrm{min}}$ and~$f_{\mathrm{max}}$ denote the minimum and maximum values attained by~$f$, respectively. We write
\[\mathsf{C}_h := |\mathcal{B}_\mathbb{T}(o,2h)| = 1 + d + \cdots + d^{2h}.\]
Finally, recall the definition of~$\Lambda(G,\Psi)$ from~\eqref{eq_def_of_Lamb}.

\begin{lemma}\label{lem_derivative}
	There exists~$\varepsilon_1 > 0$ and~$\delta > 0$ such that for any~$t \ge 0$, on the event~$\{\Lambda(G_t,\Psi_t) \ge (1-\varepsilon_1)\frac{dn}{2}\}$, we have
	\begin{equation*}
		\left.\frac{\mathrm{d}}{\mathrm{d}s} \mathbb{E}[\exp(-\delta \cdot X_{t+s})\mid \mathcal{F}_t] \right|_{s=0+} \le -  \frac{\delta \cdot \mu}{4}\cdot X_t\cdot\exp(-\delta \cdot X_t).
	\end{equation*}
\end{lemma}
\begin{proof}
	By the Markov property, it suffices to treat~$t = 0$. We let~$\varepsilon_1 > 0$ and~$\delta > 0$ be  small constants to be chosen later, and fix~$(G_0, \Psi_0) = (G,\Psi)$ such that~$\Lambda(G,\Psi) \ge (1-\varepsilon_1)\frac{dn}{2}$. We write~$X = X_0$.

	Let~$\mathcal{N}_\mathrm{cont}$ be the set of pairs~$(\tilde{G},\tilde{\Psi})$ that can be reached from~$(G,\Psi)$ with a single jump of the dynamics of~$\{(G_t,\Psi_t)\}$ of the type ``contact death'' or ``contact birth''. For each~$(\tilde{G},\tilde{\Psi})$, let~$r(\tilde{G},\tilde{\Psi})$ denote the rate at which~$(G,\Psi)$ jumps to~$(\tilde{G},\tilde{\Psi})$; note that~$r(\tilde{G},\tilde{\Psi}) \in \{1,\lambda,\ldots, d\lambda\}$.

	Next, let~$\mathcal{N}_\mathrm{sw}$ be the set of pairs~$(\tilde{G},\tilde{\Psi})$ that can be reached from~$(G,\Psi)$ with a single jump of the dynamics of~$\{(G_t,\Psi_t)\}$ associated to a switch. We decompose
	\[\mathcal{N}_\mathrm{sw} = \mathcal{N}_\mathrm{neutral} \cup \mathcal{N}_\mathrm{inactive} \cup \mathcal{N}_\mathrm{active} \cup \mathcal{N}_\mathrm{bad},\]
	according to the type of switch (with respect to~$(G,\Psi)$) that corresponds to the jump from~$(G,\Psi)$ to~$(\tilde{G},\tilde{\Psi})$, as categorized in Section~\ref{ss_def_of_emb}.

	We abbreviate, for~$(\tilde{G},\tilde{\Psi}) \in \mathcal{N}_\mathrm{cont} \cup \mathcal{N}_\mathrm{sw}$,
	\[\tilde{X} = \tilde{X}(\tilde{G},\tilde{\Psi}) = \sum_{i \in \tilde{\mathcal{J}}} f(\tilde{B}^i,\tilde{\beta}^i),\quad \text{where} \quad \tilde{\Psi} = (\tilde{\mathcal{J}},\{(\tilde{B}^i,\tilde{\beta}^i):i \in \tilde{\mathcal{J}}\}).\]
We can now write
	\begin{equation}\begin{split}\left.\frac{\mathrm{d}}{\mathrm{d}s} \mathbb{E}[\exp(-\delta \cdot X_{s})] \right|_{s=0+} 
	=&\sum_{(\tilde{G},\tilde{\Psi}) \in \mathcal{N}_\mathrm{cont}} \; r(\tilde{G},\tilde{\Psi})\cdot (e^{-\delta \cdot \tilde{X} }-e^{-\delta \cdot X}) \\[.2cm]&+ \sum_{(\tilde{G},\tilde{\Psi}) \in \mathcal{N}_\mathrm{sw}} \upupsilon_n\cdot (e^{-\delta \cdot \tilde{X}} - e^{-\delta \cdot X}).\end{split}\label{eq_auxiliary_with_Psi}\end{equation}
	Note that in the last sum above, we can discard the pairs~$(\tilde{G},\tilde{\Psi}) \in\mathcal{N}_\mathrm{neutral} \cup \mathcal{N}_\mathrm{inactive}$, as for these cases we have~$\tilde{X} = X$. Next, defining ~$\mathscr{E}(z) := e^{-z} + z - 1$ for~$z \in \mathbb{R}$ we have
	\[e^{-\tilde{z}}- e^{-z} = e^{-z}(e^{-(\tilde{z}-z)}-1) =- e^{-z}(\tilde{z}-z + \mathscr{E}(\tilde{z}-z));\]
		using this in the right-hand side of~\eqref{eq_auxiliary_with_Psi}, we obtain
	\begin{equation}\label{eq_second_for_der}
		\left.\frac{\mathrm{d}}{\mathrm{d}s} \mathbb{E}[\exp(-\delta \cdot X_{s})] \right|_{s=0+} = -\delta \cdot e^{-\delta \cdot X} \cdot (\mathcal{S} + \mathcal{S}' + \mathcal{S}''),
	\end{equation}
where
	\begin{align*}
		&\mathcal{S} := \sum_{(\tilde{G},\tilde{\Psi}) \in \mathcal{N}_\mathrm{cont}} r(\tilde{G},\tilde{\Psi})\cdot (\tilde{X} - X) + \upupsilon_n \cdot \sum_{(\tilde{G},\tilde{\Psi}) \in \mathcal{N}_\mathrm{active}}(\tilde{X} - X),\\[.2cm]
		&\mathcal{S}' := \upupsilon_n \cdot \sum_{(\tilde{G},\tilde{\Psi}) \in \mathcal{N}_\mathrm{bad}}(\tilde{X} - X),\\[.2cm]	
		&\mathcal{S}'':= \sum_{(\tilde{G},\tilde{\Psi}) \in \mathcal{N}_\mathrm{cont}} r(\tilde{G},\tilde{\Psi})\cdot \frac{\mathscr{E}(\delta\cdot(\tilde{X} - X))}{\delta} + \upupsilon_n \cdot \sum_{(\tilde{G},\tilde{\Psi}) \in \mathcal{N}_\mathrm{active}\cup \mathcal{N}_\mathrm{bad}}\frac{\mathscr{E}(\delta\cdot(\tilde{X} - X))}{\delta}.
	\end{align*}
		
		We now make the key observation that
		\begin{align*}
			\mathcal{S} =\sum_{i \in \mathcal{J}} \left( S_\mathrm{birth}(B^i,\beta^i) + S_\mathrm{death}(B^i,\beta^i) + \sum_{e \in \Lambda(G,\Psi)} S_\mathrm{switch}(B^i,\beta^i,e)\right),
		\end{align*}
		where we have employed the notation introduced in~\eqref{eq_death_B},~\eqref{eq_birth_B} and~\eqref{eq_def_S_switch}. By Lemma~\ref{lem_single_herd}, if~$\varepsilon_1 \le \varepsilon_0$ and~$\Lambda(G,\Psi) >(1- \varepsilon_1)\frac{dn}{2}$, then
	\begin{equation}\label{eq_with_S}	
			\mathcal{S} \ge \sum_{i \in \mathcal{J}} \frac{\mu}{2}\cdot f(B^i,\beta^i) = \frac{\mu}{2}\cdot X.
	\end{equation}

	We now claim that
	\begin{equation}\label{eq_first_last}
		|\mathcal{S}'| \le \frac14\cdot \mathcal{S} \quad \text{and}\quad |\mathcal{S}''| \le \frac14\cdot \mathcal{S}.
	\end{equation}
Together with~\eqref{eq_second_for_der} and~\eqref{eq_with_S}, this will complete the proof. In order to prove this claim, let us start with a simple observation. 
	Since any jump from~$(G,\Psi)$ affects (possibly erasing) at most two pairs $(B^i,\beta^i)$, and brings in at most two new pairs~$(\tilde{B}^i,\tilde{\beta}^i)$, we have
	\begin{equation}
		\label{eq_bound_diffX} |\tilde{X} - X| \le 4f_\mathrm{max}\quad \text{for all } (\tilde{G},\tilde{\Psi}) \in \mathcal{N}_\mathrm{cont} \cup \mathcal{N}_\mathrm{active} \cup \mathcal{N}_\mathrm{bad}.
	\end{equation}
	This gives
	\[|\mathcal{S}'| \le 4f_\mathrm{max} \cdot \upupsilon_n \cdot |\mathcal{N}_\mathrm{bad}|.\] Next, any bad switch must involve one edge in~$\cup_{i \in \mathcal{J}} B^i$, one edge of~$G$ that is not in~$\Lambda(G,\Psi)$, and one parity signal~$\sigma$. Also using the fact that any~$B^i$ has at most~$\mathsf{C}_h$ edges, we obtain the bound
	\[|\mathcal{N}_\mathrm{bad}| \le 2\mathsf{C}_h\cdot |\mathcal{J}| \cdot (\tfrac{nd}{2}-|\Lambda(G,\Psi)|) \le  2\mathsf{C}_h\cdot |\mathcal{J}| \cdot \varepsilon_1 \cdot \tfrac{nd}{2}.\]
	Combining these  inequalities and using~$\upupsilon_n = \tfrac{\mathsf{v}}{nd}$ gives
	\[|\mathcal{S}'| \le 4f_\mathrm{max}\cdot \mathsf{v} \cdot \mathsf{C}_h \cdot  \varepsilon_1\cdot |\mathcal{J}|  \]
	Noting  that~\eqref{eq_with_S} and the definition of~$X$ give~$\mathcal{S} \ge \tfrac{\mu}{2} \cdot f_\mathrm{min}\cdot |\mathcal{J}|$, we obtain that, if~$\varepsilon_1 < \frac{\mu\cdot f_\mathrm{min}}{32f_\mathrm{max}\cdot \mathsf{v}\cdot  \mathsf{C}_h}$, then the first inequality in~\eqref{eq_first_last} holds.

	It remains to prove the second inequality in~\eqref{eq_first_last}. Using the definition of~$\mathscr{E}$ and~\eqref{eq_bound_diffX}, we have that there exists~$C > 0$ and~$\delta_0 > 0$  such that, if~$\delta < \delta_0$,
	\[\frac{|\mathscr{E}(\delta\cdot (\tilde{X}-X))|}{\delta} \le C\delta\quad \text{for all } (\tilde{G},\tilde{\Psi}) \in \mathcal{N}_\mathrm{cont} \cup \mathcal{N}_\mathrm{active} \cup \mathcal{N}_\mathrm{bad}.
\]
	We thus obtain, for any~$\delta < \delta_0$,
	\[|\mathcal{S}''| \le C\delta\cdot \left((d\lambda + 1)\cdot |\mathcal{N}_\mathrm{cont}| + \upupsilon_n \cdot |\mathcal{N}_\mathrm{active} \cup \mathcal{N}_\mathrm{bad}|\right).\]
	We have the easy bounds
	\begin{equation}\label{eq_easy_one}|\mathcal{N}_\mathrm{cont}| \le 2\mathsf{C}_h \cdot |\mathcal{J}|.\end{equation}
and
	\begin{equation}\label{eq_easy_two}|\mathcal{N}_\mathrm{active} \cup \mathcal{N}_\mathrm{bad}| \le |\mathcal{N}_\mathrm{sw}|\le 2\mathsf{C}_h\cdot |\mathcal{J}|\cdot \frac{nd}{2},\end{equation} 
	the latter bound coming from choosing an edge of~$\cup_i B^i$ and an arbitrary other edge of~$G$. Again using~$\upupsilon_n = \frac{\mathsf{v}}{nd}$, we then get, for some constant~$C' > 0$,
	\[|\mathcal{S}''| \le C'\delta \cdot |\mathcal{J}|. \]
	Again noting that~\eqref{eq_with_S} gives~$\mathcal{S} \ge \tfrac{\mu}{2}\cdot f_\mathrm{min} \cdot |\mathcal{J}|$, we can choose~$\delta$ small enough that the desired second inequality in~\eqref{eq_first_last} holds.
\end{proof}

\begin{lemma}\label{lem_first_super}
Define the stopping times
	\begin{align}\label{eq_frak_lambda}&\mathfrak{t}_\Lambda := \inf\{t \ge 0:\Lambda(G_t,\Psi_t) < (1-\varepsilon_1)dn\},\\
		\label{eq_frak_low}&\mathfrak{t}_\mathrm{low} := \inf\left\{t \ge 0: X_t \le \frac{X_0}{2}\cdot e^{\frac{\mu}{4}\cdot t} \right\}
\end{align}
and the process
	\[Y_t:= \exp\left(\delta\cdot \left( \frac{X_0}{2} \cdot \exp\left( \frac{\mu}{4}\cdot t\right) -  X_t\right)\right),\qquad t\ge 0,\]
where~$\varepsilon_1$ and~$\delta$ are the constants given in Lemma~\ref{lem_derivative}. 
	We then have that~$(Y_{t \wedge \mathfrak{t}_\Lambda \wedge \mathfrak{t}_\mathrm{low}})_{t \ge 0}$ is a supermartingale. Moreover,
	\[ \mathbb{P}\left( \mathsf{t}_\mathrm{low} < \infty,\; \mathsf{t}_\mathrm{low} < \mathsf{t}_\Lambda\right) < \exp(-\tfrac{\delta}{2}\cdot X_0). \]
\end{lemma}
\begin{proof}
	For any~$t \ge 0$, on the event~$\{\mathfrak{t}_\Lambda \wedge \mathfrak{t}_\mathrm{low} > t\}$ we have
	\begin{align*}
		\left. \frac{\mathrm{d}}{\mathrm{d}s} \mathbb{E}[Y_{t+s}\mid \mathcal{F}_t]\right|_{s = 0+} =& \left. \frac{\mathrm{d}}{\mathrm{d}s} \exp\left(\tfrac12\cdot {\delta \cdot X_0\cdot e^{\frac{\mu}{4}\cdot (t+s)}} \right) \right|_{s=0+}\cdot e^{-\delta X_t} \\[.2cm]&+ \exp\left(\tfrac12\cdot {\delta \cdot X_0\cdot e^{\frac{\mu}{4}t}}\right)\cdot \left. \frac{\mathrm{d}}{\mathrm{d}s} \mathbb{E}[e^{-\delta \cdot X_{t+s}}\mid \mathcal{F}_t]\right|_{s = 0+}\\[.2cm]
		&\le \left(\tfrac18\cdot{\delta \cdot \mu \cdot X_0}\cdot e^{\frac{\mu}{4}\cdot t}- \tfrac14 \cdot\delta\cdot \mu \cdot X_t\right) \cdot Y_t \le 0,
	\end{align*}
	where in the first inequality we used Lemma~\ref{lem_derivative} and the assumption that~$t < \mathfrak{t}_\Lambda$, and in the second inequality we used the assumption that~$t < \mathfrak{t}_\mathrm{low}$. This proves the first statement.
	For the second statement, the optional stopping theorem gives
	\begin{align*}
		\exp(-\delta\cdot X_0/2) =Y_0 \ge \mathbb{E}[Y_{\mathsf{t}_\mathrm{low}}\cdot \mathds{1}\{\mathsf{t}_\mathrm{low} < \infty,\; \mathsf{t}_\mathrm{low} < \mathsf{t}_\Lambda\}] \ge \mathbb{P}(\mathsf{t}_\mathrm{low} < \infty,\; \mathsf{t}_\mathrm{low} < \mathsf{t}_\Lambda).
	\end{align*}
\end{proof}

\begin{lemma}\label{lem_new_super}
	There exists~$\bar{\mu} > 0$ (depending only on~$d,\lambda,\mathsf{v},h$) such that
	\[
		\left. \frac{\mathrm{d}}{\mathrm{d}s}\mathbb{E}[e^{X_{t+s}}\mid \mathcal{F}_t]\right|_{s=0+} \le \bar{\mu}\cdot X_t \cdot e^{X_t}.
	\]
\end{lemma}
\begin{proof}
We again only treat~$t= 0$. Recalling the notation from the proof of Lemma~\ref{lem_derivative}, we have
	\begin{equation}
		\left. \frac{\mathrm{d}}{\mathrm{d}s}\mathbb{E}[e^{X_s}]\right|_{s = 0+} = e^{X} \cdot \left(\sum_{(\tilde{G},\tilde{\Psi}) \in \mathcal{N}_\mathrm{cont}} r(\tilde{G},\tilde{\Psi})\cdot (e^{\tilde{X}-X}-1) + \upupsilon_n\cdot \sum_{(\tilde{G},\tilde{\Psi}) \in \mathcal{N}_\mathrm{active}\cup \mathcal{N}_\mathrm{bad}}(e^{\tilde{X}-X}-1)\right).\label{eq_with_tildeX}
	\end{equation}
	Recall that~$r(\tilde{G},\tilde{\Psi}) \le (d\lambda + 1)$ for~$(\tilde{G},\tilde{\Psi}) \in \mathcal{N}_\mathrm{cont}$, and let
	\[ K:= e^{4f_\mathrm{max}}+1 \stackrel{\eqref{eq_bound_diffX}}{\ge} e^{\tilde{X}-X}-1 \quad \text{for all } (\tilde{G},\tilde{\Psi}) \in \mathcal{N}_\mathrm{cont} \cup \mathcal{N}_\mathrm{active} \cup \mathcal{N}_\mathrm{bad}.\]
	The expression in~\eqref{eq_with_tildeX} is then smaller than
	\begin{align*}&e^X\cdot K \cdot \left((d\lambda + 1)\cdot |\mathcal{N}_\mathrm{cont}|+\tfrac{\mathsf{v}}{dn} \cdot |\mathcal{N}_\mathrm{active} \cup \mathcal{N}_\mathrm{bad}|\right)\\
	&\stackrel{\eqref{eq_easy_one}, \eqref{eq_easy_two}}{\le} e^X\cdot K \cdot \left((d\lambda + 1)\cdot 2\mathsf{C}_h\cdot |\mathcal{J}| + \tfrac{\mathsf{v}}{dn}\cdot 2\mathsf{C}_h \cdot |\mathcal{J}|\cdot \tfrac{nd}{2} \right).\end{align*}
	Noting that~$|\mathcal{J}| \le \frac{1}{f_\mathrm{min}}\cdot X$, the proof is complete.
\end{proof}

\begin{lemma}\label{lem_neww_super}
 Defining the stopping time
	\begin{equation}\label{eq_frak_high}\mathsf{t}_\mathrm{high} := \inf\left\{t \ge 0:\; X_t \ge 2X_0\cdot e^{\bar{\mu}\cdot t}\right\},\end{equation}
	we have
	\[\mathbb{P}\left(\mathsf{t}_\mathrm{high} < \infty\right) < \exp(-X_0).\]
\end{lemma}
\begin{proof}
	Let~$Z_t:= \exp(X_t - 2X_0\cdot e^{\bar{\mu}\cdot t})$ for~$t \ge 0$. A computation similar to the one carried out in the proof of Lemma~\ref{lem_first_super}, this time using Lemma~\ref{lem_new_super}, shows that~$(Z_t)$ is a supermartingale. Then, the optional stopping theorem gives
	\[\exp(-X_0) = Z_0 \ge \mathbb{E}[Z_{\mathsf{t}_\mathrm{high}}\cdot \mathds{1}\{\mathsf{t}_\mathrm{high} < \infty\}] \ge \mathbb{P}(\mathsf{t}_\mathrm{high} < \infty).\]	
\end{proof}

Recall that~$\ell_m(G)$ denotes the number of loops of length at most~$m$ in~$G$.
\begin{proposition}\label{prop_with_eps2}
	There exists~$t_0 > 0$ and~$\varepsilon_2 > 0$ such that if~$X_0 = X(G_0,\Psi_0) \le \varepsilon_2\cdot n$, then
	\begin{equation*}
		\mathbb{P}\left(\{X_{t_0} \ge  X_0\}\cup \{\ell_h(G_t) > \varepsilon_2 \cdot n \text{ for some } t \le t_0\}  \right) \ge 1-\exp(-X_0)-\exp(-\delta X_0),
	\end{equation*}
	where~$\delta$ is given in Lemma~\ref{lem_derivative}.
\end{proposition}
\begin{proof}
	Recall the constant~$\varepsilon_1$ given in Lemma~\ref{lem_derivative}. We start by choosing~$\varepsilon_1'$ small enough that, if~$(G,\Psi)$ is such that~$\ell_h(G) \le \varepsilon_1'\cdot n$ and~$X(G,\Psi) \le \varepsilon_1'\cdot n$, then~$\Lambda(G,\Psi) > (1-\varepsilon_1)\cdot \frac{dn}{2}$. This is possible, since the number of edges of~$G$ that are not in~$\Lambda(G,\Psi)$ is smaller than~$C(\ell_h(G) + |\mathcal{J}|)\le C(\ell_h(G) + \frac{1}{f_\mathrm{min}}\cdot X)$, where~$C$ is a constant that depends only on~$d,\lambda,\mathsf{v},h$. 

	Next, choose~$t_0$ large enough that
	\[\tfrac12 \cdot \exp(\tfrac{\mu}{4}\cdot t_0) > 1,\]
	and choose~$\varepsilon_2$ small enough  that
	\[2\varepsilon_2 \cdot \exp({\bar{\mu}\cdot t_0}) < \varepsilon_1'.\]

	Define the stopping time
	\[\mathfrak{t}_\mathrm{loop}:= \inf\{t \ge 0:\;\ell_h(G_t) \ge \varepsilon_2\cdot n\}.\]
	Recalling the definition of~$\mathfrak{t}_\Lambda$ in~\eqref{eq_frak_lambda} and the definition of~$\mathfrak{t}_\mathrm{high}$ in~\eqref{eq_frak_high}, we have
	\[\mathfrak{t}_\Lambda \ge \mathfrak{t}_\mathrm{high} \wedge \mathfrak{t}_\mathrm{loop}.\]
	Recalling the definition of~$\mathfrak{t}_\mathrm{low}$ in~\eqref{eq_frak_low}, by the choice of~$t_0$ we have
	\[\{\mathfrak{t}_\mathrm{low} > t_0\} \subseteq \{X_{t_0} \ge \tfrac12 X_0\cdot \exp(\tfrac{\mu}{4}\cdot t_0)\}\subseteq \{ X_{t_0} \ge X_0\}.\]
	We can then bound
	\begin{align*}
		\mathbb{P}(\{X_{t_0} < X_0\} \cap \{\mathfrak{t}_\mathrm{loop} > t_0\})&\le \mathbb{P}(\mathfrak{t}_\mathrm{low} \le t_0,\; \mathfrak{t}_\mathrm{loop} > t_0)\\
		&\le \mathbb{P}(\mathfrak{t}_\mathrm{low} \le t_0,\; \mathfrak{t}_\mathrm{high} > t_0,\;\mathfrak{t}_\mathrm{loop} > t_0) + \mathbb{P}(\mathfrak{t}_\mathrm{high} < \infty)\\
		&\le\mathbb{P}(\mathfrak{t}_\mathrm{low} \le t_0,\; \mathfrak{t}_\Lambda > t_0) + \mathbb{P}(\mathfrak{t}_\mathrm{high} < \infty).
	\end{align*}
	Finally, the two probabilities on the right-hand side are bounded using Lemmas~\ref{lem_first_super} and~\ref{lem_neww_super}.
\end{proof}

\begin{proof}[Proof of Theorem~\ref{thm_main}]
	We fix a constant~$\varepsilon_3 > 0$ with
	\[\varepsilon_3 < \min\left(\frac{\varepsilon_2}{2},\; \frac{f_\mathrm{min}}{2\mathsf{C}_{2h}}\right),\] where~$\varepsilon_2$ is the constant of Proposition~\ref{prop_with_eps2}.  Given~$G \in \mathcal{G}_n$ and a set~$\xi$ of vertices of~$G$, we say that the pair~$(G,\xi)$ is \textit{good} if there exists an indexed set~$\Psi = (\mathcal{J},\{(B^i,\beta^i)\})$ of~$h$-herds of~$G$ with the property that~$\beta^i \subset \xi$ for all~$i \in \mathcal{J}$ and~$X(G,\Psi) \ge \varepsilon_3\cdot n$. We now prove some claims involving this definition.
	
	\begin{claim}\label{cl_round} If~$n> \frac{2f_\mathrm{max}}{\varepsilon_2}$ and~$(G,\xi)$ is good, then there exists an indexed set~$\Psi = (\mathcal{J},\{(B^i,\beta^i)\})$ of~$h$-herds of~$G$ with~$\beta^i \subset \xi$ for all~$i \in \mathcal{J}$ and~$X(G,\Psi) \in [\varepsilon_3 \cdot n,\; \varepsilon_2\cdot n]$.
	\end{claim}

	\begin{proof}
		Let~$\Psi$ be an indexed set of~$h$-herds of~$G$ obtained from the definition of~$(G,\xi)$ being good. We have~$X(G,\Psi) \ge \varepsilon_3 \cdot n$, and 
in case~$X(G,\Psi) > \varepsilon_2\cdot n$, we can remove~$h$-herds  from~$\Psi$ one by one until the value of~$X$ drops below~$\varepsilon_2 \cdot n$ (the removal of an~$h$-herd decreases the value of~$X$ by at most~$f_\mathrm{max}$, so since~$\tfrac{\varepsilon_2}{2} \cdot n > f_\mathrm{max}$, we can indeed end up with~$X$ between~$[\tfrac{\varepsilon_2}{2}\cdot n,\;\varepsilon_2\cdot n]$).
	\end{proof}

	\begin{claim}\label{cl_to_markov}
		Assume that~$n > \frac{2f_\mathrm{max}}{\varepsilon_2}$,~$(G,\xi)$ is good,~$(G_t)$ is a switching graph with~$G_0 = G$ and~$(\xi_t)$ is the contact process on this graph with~$\xi_0 = \xi$. Then,
		\[
			\mathbb{P}(\{(G_{t_0},\xi_{t_0}) \text{ is good}\} \cup \{\ell_h(G_t) > \varepsilon_2 \cdot n \text{ for some } t \le t_0\}) > 1 - \exp(-\varepsilon_3 \cdot n) - \exp(-\delta \cdot \varepsilon_3\cdot n).
		\]
	\end{claim}
	\begin{proof}
		Fix an indexed set~$\Psi = (\mathcal{J},\{(B^i, \beta^i)\})$ of~$h$-herds in~$G$ as given by Claim~\ref{cl_round}.  It is easy to see that we can construct in a single probability space a process~$(G_t,\xi_t, {\Psi}_t)_{t \ge 0}$, where
		\begin{itemize}
			\item $(G_t)$ is a switching graph with~$G_0 = G$;
			\item $(\xi_t)$ is a contact process on~$(G_t)$ with~$\xi_0 = \xi$;
			\item $({\Psi}_t)$ is an~$h$-herds process on~$(G_t)$ with~${\Psi}_0 = {\Psi}$;
		\end{itemize}
		and moreover,
		\[\bigcup_{i \in {\mathcal{J}}_t} \beta^i_t \subset \xi_t  \text{ for all } t \ge 0, \quad \text{where } {\Psi}_t = ({\mathcal{J}_t},\{(B^i_t,\beta^i_t):i \in {\mathcal{J}}_t\}).\]
Under this coupling we have
		\[\{X(G_{t_0},{\Psi}_{t_0}) \ge X(G_0,\Psi_0) \}\supset \{(G_{t_0},\xi_{t_0}) \text{ is good}\}. \]
		The statement of the claim now readily follows  from Proposition~\ref{prop_with_eps2}, which can be applied since~$X(G,\Psi) < \varepsilon_2 \cdot n$.
	\end{proof}

	\begin{claim}
		If~$\ell_h(G) < \frac{1}{2h\cdot \mathsf{C}_{h}}\cdot n$ and~$\xi$ is the set of all vertices of~$G$, then~$(G,\xi)$ is good.
	\end{claim}
	\begin{proof}
		Given a loop in~$G$ with length at most~$h$, there are fewer than~$h\cdot \mathsf{C}_{h}$ vertices~$x$ in~$G$ so that~$\mathcal{B}_G(x,h)$ intersects this loop. Therefore, if~$\ell_h(G) < \frac{1}{2h\cdot \mathsf{C}_{h}}\cdot n$, then there exists a set of vertices~$V_0$ of~$G$ with~$|V_0| \ge \frac{n}{2}$ and so that~$\mathcal{B}_G(x,h)$ is isomorphic to~$\mathcal{B}_\mathbb{T}(o,h)$ for all~$x \in V_0$. Next, given~$x \in V_0$, there are fewer than~$\mathsf{C}_{h}$ vertices~$y \in V_0$ such that~$\mathcal{B}_G(x,h) \cap \mathcal{B}_G(y,h) \neq \varnothing$. It is then an easy combinatorial exercise to show that we can obtain a set~$V_1 \subset V_0$ with~$|V_1| \ge \frac{1}{\mathsf{C}_{h}}\cdot |V_0| \ge \frac{n}{2\mathsf{C}_{h}}$ and so that the balls~$\{\mathcal{B}_G(x,h):x \in V_1\}$ are all disjoint. We now let~$\Psi = (\mathcal{J},\{(B^i,\beta^i): i \in \mathcal{J}\})$, where
		\[\mathcal{J} = V_1\quad \text{and}\quad B^x = \mathcal{B}_G(x,h),\; \beta^x = \{x\} \text{ for all } x \in V_1.\]
		We then have~$X(G,\Psi) \ge f_\mathrm{min}\cdot |V_1| \ge \frac{f_\mathrm{min}}{2\mathsf{C}_{h}}\cdot n \ge \varepsilon_3 \cdot n$, so~$(G,\xi)$ is good.
	\end{proof}
	We are now ready to conclude. Assume that~$(G_t)$ is a switching graph started from the stationary distribution and~$(\xi_t)$ is a contact process on~$(G_t)$ started from full  occupancy. For any~$k \in \mathbb{N}$ we bound
	\begin{align*}
		\mathbb{P}(\xi_{k\cdot t_0} = \varnothing)& \le \mathbb{P}\left(\ell_h(G_0) \ge \tfrac{n}{2h\cdot \mathsf{C}_{h}}\right) + \mathbb{P}\left(\xi_{k\cdot t_0} = \varnothing,\;\ell_h(G_0) < \tfrac{n}{2h\cdot \mathsf{C}_{h}}\right) \\
		&\le \mathbb{P}\left(\ell_h(G_0) \ge \tfrac{n}{2h\cdot \mathsf{C}_{h}}\right)+\mathbb{P}\left(\exists t\le k\cdot t_0:\; \ell_h(G_t) > \varepsilon_2\cdot n\right)\\
		&\hspace{5cm} + k \left(\exp(-\varepsilon_3\cdot n) + \exp(-\delta \cdot \varepsilon_3\cdot n)\right),
	\end{align*}
	where in the second inequality we used a union bound combined with Claim~\ref{cl_to_markov} and the Markov property. Using Proposition~\ref{prop_no_loops}, the desired statement now follows by taking~$k = \exp(c'\cdot n)$, with~$c'$ sufficiently small.
\end{proof}

\section{Strict decrease of critical value: Proof of Theorem~\ref{thm_strict}}\label{sec_freeze}
In this final section we give the proof of Theorem~\ref{thm_strict}, showing that the critical value of the herds process~$(\Xi_t)_{t \ge 0}$ is strictly smaller than~$\lambda_c(\mathbb{T})$, the threshold between extinction and (global) survival for the contact process on the (static)~$d$-regular tree. The contents of this section are independent of Sections~\ref{s_hherds} and~\ref{s_embed}.

\subsection{Auxiliary processes: marked particles and freezing}\label{ss_marked} We now define two variants of the herds process.  The first one, denoted~$(\Upsilon_t)_{t \ge 0}$, incorporates marked particles, which are deemed higher than normal particles in a hierarchical relation governing the births. The second one, denoted~$(\dbh_t)_{t \ge 0}$, also includes marked particles, and on top of that it includes a mechanism called \textit{freezing} of herds.\medskip

\noindent \textbf{Marked particles.} We now introduce a dynamics similar to the herds process, with the difference that each particle is either marked or normal. Thus, a state of the process is denoted by
\[\Upsilon = (\mathcal{J},\{(\eta^i, \xi^i):i \in \mathcal{J}\}),\]
where~$\mathcal{J} \subset \mathbb{N}$ is the set of herds and for each~$i\in \mathcal{J}$,~$\eta^i$ is the set of all particles in herd~$i$ and~$\xi^i \subset \eta^i$ is the set of marked particles in herd~$i$. 

The rules of evolution will guarantee that the following \textit{exclusion rule} is always satisfied: \textit{for each~$u \in \mathbb{T}$, there is at most one herd in which there is a marked particle at~$u$}. Using this property, we will conclude that the process~$(\cup_{i \in \mathcal{J}_t} \xi^i_t)_{t \ge 0}$ is a contact process on~$\mathbb{T}$. Hence, the purpose of introducing this process is to obtain a coupling between the herds process and the contact process on~$\mathbb{T}$ with the same infection rate, in which the former dominates the latter.

Let us now explain the dynamics of~$(\Upsilon_t)_{t\ge 0}$. As in the herds process, all particles (marked or normal) die with rate one. Active edges in each herd split with rate~$\mathsf{v}$, with the following points to note:
\begin{itemize}
	\item as before, an edge is deemed active in a herd if it is a bridge between two non-empty portions of~$\mathbb{T}$ in that herd, where `non-empty' means containing particles (marked or normal);
	\item the effect of an edge splitting is just as before: the herd in question is replaced by two new herds, each containing the particles (marked or normal) that were on each of the sides of the splitting edge.
\end{itemize}

The birth rules of particles are a little more involved. All particles, marked or normal, try to give birth at each neighboring position in their herd with rate~$\lambda$. Then, we can summarize the outcome  of each attempt by the following:
\begin{itemize} \item normal particles try to create normal particles, but are not allowed to do so on top of existing marked particles; 
		\item marked particles try to create marked particles (even overwriting existing particles), but if this creation would mean violating the exclusion rule, then they create a normal particle instead.
\end{itemize}
More formally, assume that a particle at~$u$ attempts to give birth at a neighboring site~$v \sim u$ in herd~$i$; then, the outcome of the attempt is given by the following table, where the uncolored border cells represent the different cases concerning~$u$ and~$v$, and the colored cells express the outcome of the attempt.
\begin{table}[H]
	\centering
	\begin{tabular}{cc|c|c|}
		\hhline{~~--} & & $u \in \xi^i$ & $u \in \eta^i \backslash\xi^i$\\[.05cm]\hhline{--==}
		\multicolumn{1}{|l|}{\multirow{3}{*}{$v \notin \cup_{j \neq i}\xi^j$ }} &  $v \in \xi^i$ & \cellcolor[HTML]{FFCFCF}no effect &\cellcolor[HTML]{FFCFCF}no effect  \\
		 \hhline{~|-|-|-} 
\multicolumn{1}{|l|}{} & $v \in \eta^i \backslash \xi^i$ &\cellcolor[HTML]{BEFFBE}effect: $v$ enters $\xi^i$ &\cellcolor[HTML]{FFCFCF}no effect  \\
 \hhline{~|-|-|-} 
		\multicolumn{1}{|l|}{} & $v \notin \eta^i$  & \cellcolor[HTML]{BEFFBE}effect: $v$ enters~$\xi^i$ & \cellcolor[HTML]{FFF2C2}effect: $v$ enters $\eta^i \backslash \xi^i$ \\ \hhline{====} 
		\multicolumn{1}{|l|}{\multirow{2}{*}{\parbox{1.9cm}{\vspace{-0.1cm}$v \in \cup_{j \neq i}\xi^j$ }}} & $v \in \eta^i \backslash \xi^i$ &  \cellcolor[HTML]{FFCFCF} no effect& \cellcolor[HTML]{FFCFCF} no effect \\  \hhline{~|-|-|-} \multicolumn{1}{|l|}{}                  & $v \notin \eta^i$ & \cellcolor[HTML]{FFF2C2}effect: $v$ enters~$\eta^i \backslash \xi^i$ & \cellcolor[HTML]{FFF2C2}effect: $v$ enters~$\eta^i \backslash \xi^i$ \\ \hline \end{tabular}
\end{table}

We now state our coupling result.
\begin{proposition}\label{prop_coupling}
	Let~$(\Upsilon_t)_{t \ge 0} = (\mathcal{J}_t,\{(\eta^i_t,\xi^i_t):i \in \mathcal{J}_t\})$ denote the herds process with marked particles with parameters~$\mathsf{v}$ and~$\lambda$. Then, the process~$(\cup_{i \in \mathcal{J}_t} \xi^i_t)_{t \ge 0}$ is a contact process on~$\mathbb{T}$ with parameter~$\lambda$, and~$(\mathcal{J}_t,\eta^i_t)_{t \ge 0}$ is a herds process with parameters~$\mathsf{v}$ and~$\lambda$.
\end{proposition}
This is proved by fixing an arbitrary configuration~$(\mathcal{J},\{(\eta^i,\xi^i):i \in \mathcal{J}\})$ and inspecting the effects of each possible transition on the `marginals'~$\cup_{i \in \mathcal{J}} \xi^i$ and~$(\mathcal{J},\{\eta^i: i \in\mathcal{J}\})$, verifying that these effects match the jump mechanisms of the marginal dynamics, with the correct rates. The details are left to the reader.\medskip

\noindent\textbf{Freezing herds.} We now introduce the process~$(\dbh_t)_{t \ge 0}$, which is a modification of the herds process with marked particles,~$(\Upsilon_t)_{t \ge 0}$. It is quite easy to explain: it evolves exactly like~$(\Upsilon_{t \ge 0})_{t \ge 0}$, with the only difference that whenever a herd contains no marked particles, it becomes \textit{frozen}, meaning that it no longer evolves in any way: its particles no longer die or give birth, and its edges no longer split.

The idea of incorporating freezing to the contact process is given in~\cite{huang}, where particles of a contact process on a finite tree become frozen when they enter certain boundary vertices. There, this technique is used to produce a process that dominates the contact process from above. Here, due to independence properties of the herds process, we  use freezing in order to provide a \textit{lower} bound for the process~$(\Upsilon_t)_{t \ge 0}$. The key idea is contained in the following result.
\begin{lemma}\label{lem_expectation_more}
	Assume that the process~$(\dbh_t)_{t \ge 0}$ has parameters~$\mathsf{v}$ and~$\lambda$, and is started with a single herd with a single marked particle. Assume that the expected number of herds that become frozen in this process is  larger than one. Then, the herds process~$(\Xi_t)_{t \ge 0}$ with parameters~$\mathsf{v}$ and~$\lambda$ survives, that is,~$\lambda > \bar{\lambda}(\mathsf{v})$.
\label{lem_freezing_exp}\end{lemma}
\begin{proof}
	Let~$K$ denote the number of herds that become frozen in the process~$(\dbh_t)_{t \ge 0}$. It is sufficient to prove the lemma under the assumption that~$K$ is almost surely finite, since otherwise~$(\dbh_t)$ clearly survives, so~$(\Xi_t)$ also survives  by Proposition~\ref{prop_coupling} and then there is nothing more to prove.

	We will define a process~$(\breve{\Upsilon}_t)_{t \ge 0}$ from~$(\dbh_t)_{t \ge 0}$; we first give an informal description. Given a trajectory of~$(\dbh_t)$, we include in~$(\breve{\Upsilon}_t)$ the unfrozen portion of~$(\dbh_t)$ (though we ignore the distinction between marked and normal particles). Moreover, for each frozen herd that appeared in the trajectory of~$(\dbh_t)$, we `unfreeze' it in~$(\breve{\Upsilon}_t)$, lettting it evolve like a herds process (with no labelling of particles as marked or normal, and no freezing); they each evolve independently, starting from the time and state in which they entered~$(\dbh_t)$.

	We now turn to a formal description. The process~$(\breve{\Upsilon}_t)$ has at a given time~$t$ a state of the form
	\[ (\breve{\mathcal{J}},\{\breve{\eta}^\ell: \ell \in \breve{\mathcal{J}}\}),\] where~$\breve{\mathcal{J}} \subset \mathbb{N}_0 \times \mathbb{N}$, and each~$\breve{\eta}^\ell \subset \mathbb{T}$. To define it, condition on a trajectory of~$(\dbh_t)_{t \ge 0} = (\mathcal{J}_t,\{(\eta^i_t,\beta^i_t): t \in \mathcal{J}_t\})_{t \ge 0}$, let~$t_1 < t_2 < \cdots$ denote the times at which frozen herds appear in this process, and for each~$k$, let~$i_k \in \mathcal{J}_{t_k}$ denote the index of the frozen herd that appears at time~$t_k$, and~$A_k \subset \mathbb{T}$ the set of particles in this frozen herd. For each~$k$, let
	\[(\Xi^{(k)}_t)_{t \ge t_k}= (\mathcal{J}^{(k)}_t,\;\{\eta^{(k,i)}_t: i \in \mathcal{J}^{(k)}_t\})_{t \ge t_k}\] denote a herds process, started from time~$t_k$ with a single herd with set~$A_k$ of particles; assume that all these herds processes are independent. We then set
	\[\breve{\mathcal{J}}_t := \{(0,i):i \in \mathcal{J}_t\} \cup \{(k,i): t_k \le t,\; i \in \mathcal{J}^{(k)}_t\}\]
	and define~$\{\breve{\eta}^\ell_t: \ell \in \breve{\mathcal{J}}_t\}$ by letting
	\begin{align*}&\breve{\eta}^{(0,i)}_t := \eta^i_t,\quad i \in \mathcal{J}_t,\\[.2cm]
		&\breve{\eta}^{(k,i)}_t := \eta^{(k,i)}_t,\quad (k,i) \in \breve{\mathcal{J}}_t,\; k\neq 0.
	\end{align*}

	We then have that, apart from a change in the set of indices, the process~$(\breve{\Upsilon}_t)_{t \ge 0}$ has the same distribution as the herds process~$(\Xi_t)_{t \ge 0}$ started from a single herd with a single particle at the root. In particular, the following holds.  For any finite~$A \subset \mathbb{T}$, let~$g(A)$ denote the probability that a herds process started from a single herd with~$A$ occupied and~$\mathbb{T}\backslash A$ vacant dies out. Then (again assuming that~$(\Xi_t)$ starts with one herd with one particle),
	\begin{align*}
		q:=\mathbb{P}((\Xi_t) \text{ dies out}) = \mathbb{E}\left[ \prod_{k=1}^K g(A_k)\right].
	\end{align*}
	Moreover, we have~$g(A_k) \le g(\{o\}) = q$ by monotonicity, so
	\begin{align*}
		\mathbb{E}\left[ \prod_{k=1}^K g(A_k) \right] \le \mathbb{E}[q^K].
	\end{align*}
	Letting
	\[G(x):= \mathbb{E}[x^K] = \sum_{k=0}^\infty x^k\cdot \mathbb{P}(K=k),\quad x \in [0,1]\]
	be the probability generating function of~$K$, the total number of frozen herds in~$(\dbh_t)$, we have thus shown that
	\[q \le G(q).\]

	Now, as in the classical proof of phase transition for branching processes, we have that~$G$ is increasing, strictly convex, has~$G(0) > 0$,~$G(1) = 1$ and~$G'(1) = \mathbb{E}[K] > 1$. This implies  that there is a unique~$x^* \in (0,1)$ such that~$G(x^*) = x^*$. The fact that~$q \le G(q)$ implies that~$q \le x^* < 1$, so~$(\Xi_t)$ survives with probability~$1-q >0$.
\end{proof}

\subsection{Strict decrease of critical value: proof of Theorem~\ref{thm_strict}}
We will use the following fact about the contact process on~$\mathbb{T}$; their proof can be found in Part~I of~\cite{lig2} (see Theorem 4.27 and discussion after Theorem 4.46).

\begin{lemma}
	Letting~$(\xi_t)_{t \ge 0}$ denote the contact process on~$\mathbb{T}$, we have that
	\begin{enumerate}
		\item if~$\lambda < \lambda_c(\mathbb{T})$, then there exist~$C,c > 0$ such that
			\begin{equation}\label{eq_sharp_threshold} \mathbb{E}[|\xi_t|] \le Ce^{-ct},\quad t \ge 0;\end{equation}
		\item if~$\lambda = \lambda_c(\mathbb{T})$, then
			\begin{equation}\label{eq_infinity_crit}
		\int_0^\infty \mathbb{E}[|\xi_t|]\;\mathrm{d}t = \infty.
	\end{equation}
	\end{enumerate}
\end{lemma}

We will also use the following simple fact about Markov chains.
\begin{lemma}\label{lem_markov}
	Let~$(X_t)_{t \ge 0}$ be a continuous-time Markov chain on a countable state space~$S$. For each~$x,y \in S$, let~$\mathrm{rate}(x,y) \ge 0$ denote the rate at which the chain jumps from~$x$ to~$y$, and  assume that~$|\{y: \mathrm{rate}(x,y) > 0\}| < \infty$ for each~$x$. For any function~$f: S \to  \mathbb{R}$ let
	\begin{equation}\label{eq_prot_gen} Lf(x):= \sum_y \mathrm{rate}(x,y)\cdot (f(y)-f(x)),\quad  x \in S.\end{equation}
	Let~$f_1,f_2:S \to \mathbb{R}$ be non-negative functions such that
	\begin{equation}\label{eq_generator_ineq} 
		Lf_2(x) \ge \alpha \cdot f_1(x) - \beta \cdot f_2(x) \quad \text{ for all } x \in S
	\end{equation}
	where~$\alpha,\beta$ are positive constants. We then have that
	\begin{equation}\label{eq_nice_exps}
		\int_0^\infty \mathbb{E}[f_2(X_t)] \;\mathrm{d}t \ge \frac{1}{\beta}\cdot \left(-\limsup_{t \to \infty}\mathbb{E}[f_2(X_t)] + \alpha \cdot   \int_0^{\infty} \mathbb{E}[f_1(X_t)]\;\mathrm{d}t \right).
	\end{equation}
\end{lemma}
\begin{proof}
	The process
	\[M_t:= f_2(X_t) - \int_0^t Lf_2(X_s)\;\mathrm{d}s,\quad t  \ge 0\]
	is a local martingale. 
For each~$a > 0$, let
	\[\tau_a:= \inf\{t \ge 0: \max(f_1(X_t),f_2(X_t),Lf_1(X_t),Lf_2(X_t)) \ge a\},\quad a > 0.\]
	By the optional  stopping theorem we have, for any~$t \ge 0$,
	\begin{align*}
		f_2(x_0)=  \mathbb{E}[M_0] &= \mathbb{E}[M_{t \wedge \tau_a}] \\&= \mathbb{E}[f_2(X_{t \wedge \tau_a})] - \int_0^{t\wedge \tau_a} \mathbb{E}[Lf_2(X_s)]\;\mathrm{d}s\\
		&\stackrel{\eqref{eq_generator_ineq}}{\le } \mathbb{E}[f_2(X_{t \wedge \tau_a})] - \alpha \cdot \int_0^{t \wedge \tau_a} \mathbb{E}[f_1(X_{s})]\;\mathrm{d}s + \beta \cdot \int_0^{t \wedge \tau_a} \mathbb{E}[f_2(X_{s})]\;\mathrm{d}s,
	\end{align*}
	so
	\[\int_0^{t \wedge \tau_a} \mathbb{E}[f_2(X_{s})]\;\mathrm{d}s \ge \frac{1}{\beta}\cdot \left( -\mathbb{E}[f_2(X_{t \wedge \tau_a})]+\alpha \cdot \int_0^{t \wedge \tau_a} \mathbb{E}[f_1(X_s)]\;\mathrm{d}s \right).\]
	The desired inequality now follows by first taking~$a \to \infty$, and then taking~$t \to \infty$.
\end{proof}

Let~$T$ be a set of vertices of~$\mathbb{T}$.
We let~$\partial_\star T$ denote the set of ordered pairs~$(u,v)$, where~$u,v$ are vertices of~$\mathbb{T}$ such that:
\begin{itemize}
	\item $u \sim v$;
	\item $u \in T$;
	\item for every~$v' \in T \backslash \{u,v\}$, the shortest path in~$\mathbb{T}$ from~$v'$ to~$v$ intersects~$u$.
\end{itemize}
Equivalently, this means that~$u \in T$,~$u \sim v$, and among the~$d$ connected components of~$\mathbb{T}$ that appear if we delete~$v$, only the one containing~$u$ intersects~$T$. Note that the property~$(u,v) \in \partial_\star T$ does not depend on whether or not~$v \in T$.
It follows from~\cite[Lemma 6.2]{pemantle} that
\begin{equation}\label{eq_from_pem}
	|\{(u,v) \in \partial_\star T: v \notin T\}| \ge  \left(1 - \frac{1}{d-1}\right)\cdot |T|.
\end{equation}

We will now define and relate several functions of a configuration~$\dbh = (\mathcal{J},\{(\eta^i,\xi^i): i \in \mathcal{J}\})$ of the process~$(\dbh_t)_{t \ge 0}$. We abuse notation and write
\[\partial_\star \dbh := \partial_\star \left(\bigcup_{i \in \mathcal{J}} \xi^i\right).\]
Then, first let
\[F_1(\dbh):= \sum_{i \in \mathcal{J}} |\xi^i|,\]
that is the number of marked particles in all herds of~$\dbh$. Second, 
\[F_2(\dbh) := \sum_{(u,v) \in \partial_\star \hat{\Upsilon}  }\; \sum_{i \in \mathcal{J}}\; \mathds{1}\{\{u,v\} \subset \xi^i\}.\]

	\begin{lemma} Assume that~$\lambda < \lambda_c(\mathbb{Z})$.
We then have
		\[\int_0^\infty \mathbb{E}[F_2(\dbh_t)]\;\mathrm{d}t \ge \frac{\lambda \cdot \left(1- \frac{1}{d-1}\right)}{2+\mathsf{v}+\lambda \cdot (d-1)} \cdot\int_0^\infty \mathbb{E}[F_1(\dbh_t)]\;\mathrm{d}t. \]
\end{lemma}
\begin{proof}
	We use Lemma~\ref{lem_markov} with~$f_1 = F_1$ and~$f_2 = F_2$. By~\eqref{eq_sharp_threshold}, Proposition~\ref{prop_coupling} and the simple bound
	\[F_2(\dbh) \le d\cdot F_1(\dbh) = d\cdot \sum_{i \in \mathcal{J}} |\xi^i| = d \cdot |\cup_{i \in \mathcal{J}} \xi^i|,\]
	we have that the~$\limsup$ that appears in~\eqref{eq_nice_exps} vanishes.
	Letting~$L$ denote the generator of the dynamics as in~\eqref{eq_prot_gen}, we will now prove that, for any state~$\dbh = (\mathcal{J},\{(\eta^i,\beta^i):i \in \mathcal{J}\})$, we have
	\begin{equation} LF_2(\dbh) \ge \lambda \cdot \left(1- \frac{1}{d-1}\right) \cdot F_1(\dbh) -(2+\mathsf{v}+\lambda \cdot (d-1))\cdot F_2(\dbh),\label{eq_case_inst}\end{equation}
		from which the statement will follow.  In order to prove~\eqref{eq_case_inst}, we examine all possible jumps of the dynamics from~$\dbh$, and how they affect the value of~$F_2$. To this end, let~$\mathrm{rate}(\dbh,\dbh')$ denote the rate at which the process jumps from~$\dbh$ to an alternate state~$\dbh'$.

	Fix~$(u,v) \in \partial_\star \dbh$ with~$v \notin \cup_{i \in \mathcal{J}} \xi^i$. Then, there exists~$i \in \mathcal{J}$ such that~$u \in \xi^i$, and moreover, if there is a birth in this herd~$i$ from~$u$ to~$v$, then a new marked particle appears there in~$v$, and the pair~$(u,v)$ then increments the value of~$F_2$ by one. This shows that
	\[\sum_{\dbh'} \mathrm{rate}(\dbh,\dbh')\cdot \max\left(F_2(\dbh') - F_2(\dbh),\;0\right)\ge \lambda \cdot |\{(u,v) \in \partial_\star \dbh: v \notin \cup_i \xi^i\}|,\]
	so by~\eqref{eq_from_pem} we obtain
	\[\sum_{\dbh'} \mathrm{rate}(\dbh,\dbh')\cdot \max\left(F_2(\dbh') - F_2(\dbh),\;0\right)\ge \lambda \cdot \left(1-\frac{1}{d-1}\right)\cdot F_1(\dbh).\]

	Next, fix~$(u,v) \in \partial_\star \dbh$ such that there is a herd~$i \in \mathcal{J}$ with~$\{u,v\} \in \xi^i$ (so that the pair~$(u,v)$ contributes to the sum that defines~$F_2(\dbh)$). Note that the following are the only jumps in the dynamics that can make it so that~$(u,v)$ no longer contributes to~$F_2$: the death of the particle at~$u$ in herd~$i$, the death of the particle at~$v$ in herd~$i$, a split of the edge between~$u$ and~$v$ in herd~$i$, or a birth from the particle at~$v$ in herd~$i$ towards a neighbor different from~$u$. This shows that
	\[\sum_{\dbh'} \mathrm{rate}(\dbh,\dbh')\cdot \min\left(F_2(\dbh') - F_2(\dbh),\;0\right)\ge -(2+\mathsf{v} + (d-1)\cdot \lambda)\cdot F_2(\dbh),\]
	completing the proof.
\end{proof}
For~$\dbh = (\mathcal{J},\{(\eta^i,\xi^i):i \in \mathcal{J}\})$, define
\[F_3(\dbh):= \sum_{(u,v) \in \partial_\star\hspace{-.08cm}\dbh} \;\sum_{i \in \mathcal{J}} \sum_{\substack{j \in \mathcal{J},\\ j \neq i}} \mathds{1}\{u \in \xi^i,\; v \notin \eta^i,\; v \in \xi^j\}.\]
\begin{lemma}
Assume that~$\lambda < \lambda_c(\mathbb{T})$. We then have
	\[\int_0^\infty \mathbb{E}[F_3(\dbh_t)]\;\mathrm{d}t \ge \frac{\mathsf{v}}{2+\lambda \cdot d} \cdot \int_0^\infty \mathbb{E}[F_2(\dbh_t)]\;\mathrm{d}t.\]
\end{lemma}
\begin{proof}
	We apply Lemma~\ref{lem_markov} with~$f_1 = F_2$ and~$f_2 = F_3$. We again have~$F_3(\dbh) \le d\cdot F_1(\dbh)$, so the~$\limsup$ in~\eqref{eq_nice_exps} vanishes. The proof will now follow from showing that
	\[LF_3(\dbh) \ge \lambda \cdot F_2(\dbh) - (2+\lambda \cdot d)\cdot F_3(\dbh).\]
	To prove this, again we inspect all the possible jumps of the dynamics from~$\dbh$, and how they affect~$F_3$. 

	Fix~$(u,v) \in \partial_\star \dbh$ such that there exists~$i \in \mathcal{J}$ such that~$u,v\in \xi^i$ (so that~$(u,v)$ contributes to the sum that defines~$F_2(\dbh)$). If the edge between~$u$ and~$v$ in herd~$i$ splits, the value of~$F_3$ is incremented by one. This shows that 
	\[\sum_{\dbh'} \mathrm{rate}(\dbh,\dbh')\cdot \max\left(F_3(\dbh')- F_3(\dbh),\;0\right) \ge \mathsf{v}\cdot F_2(\dbh).\]

	Now, fix~$(u,v) \in \partial_\star \dbh$ such that there exist distinct herd indices~$i, j \in \mathcal{J}$ such that~$u \in \xi^i$,~$v \in \xi^j$ and~$v \notin \eta^i$ (so that the pair~$(u,v)$ contributes to the sum that defines~$F_3(\dbh)$). Then, the following are the only jumps in the dynamics that can make it so that~$(u,v)$ no longer contributes to~$F_3$: (i) the death of the marked particle at~$u$ in herd~$i$, (ii) a birth from this particle towards position~$v$, (iii) the death of the marked particle at~$v$ at herd~$j$, (iv) a birth from this particle towards a neighboring position distinct from~$u$. Taking this into account we have
	\[\sum_{\dbh'} \mathrm{rate}(\dbh,\dbh')\cdot \min\left(F_3(\dbh')- F_3(\dbh),\;0\right) \ge -(2+\lambda \cdot d)\cdot F_3(\dbh).\]
\end{proof}

We now define our fourth function of a configuration~$\dbh = (\mathcal{J},\{(\eta^i,\xi^i):i \in \mathcal{J}\})$ by
\[F_4(\dbh):= \sum_{(u,v) \in \partial_\star \hspace{-.08cm}\dbh}\; \sum_{i \in \mathcal{J}} \mathds{1}\{u \in \xi^i,\;v \in \eta^i\backslash \xi^i\}.\]
The next is proved by very similar reasoning as the previous two lemmas, so we omit the details.
\begin{lemma}
	Assume that~$\lambda < \lambda_c(\mathbb{T})$. We then have
	\[\int_0^\infty \mathbb{E}[F_4(\dbh_t)]\;\mathrm{d}t \ge \frac{\mathsf{\lambda}}{2+\lambda+\mathsf{v}} \cdot \int_0^\infty \mathbb{E}[F_3(\dbh)]\;\mathrm{d}t.\]
\end{lemma}

Finally, we let~$F_5(\dbh)$ denote the number of frozen herds in~$\dbh$. 
\begin{lemma} We have
	\[\mathbb{E}[F_5(\dbh_t)] \ge \mathsf{v}\cdot \int_0^t F_4(\dbh_s)\;\mathrm{d}s.\]
\end{lemma}
\begin{proof}
	For  any state~$\dbh = (\mathcal{J},\{(\eta^i,\xi^i):i\in\mathcal{J}\})$, letting~$L$ denote the generator of the dynamics of~$(\dbh_t)_{t \ge 0}$ we have
	\[LF_5(\dbh) \ge \mathsf{v} \cdot F_4(\dbh).\]
	Indeed, let~$(u,v) \in \partial_\star \dbh$ be a pair contributing to the sum that defines~$F_4$. Then, there exists a herd~$i \in \mathcal{J}$ such that~$u \in \xi^i$ and~$v \in \eta^i\backslash \xi^i$, so a split in the edge between~$u$ and~$v$ in this herd produces a frozen herd, by the definition of~$\partial_\star \dbh$. The result now follows from noting that~$\mathbb{E}[F_5(\dbh_t)] = \int_0^t LF_5(\dbh_s)\;\mathrm{d}s$.
\end{proof}

\begin{proof}[Proof of Theorem~\ref{thm_strict}]
	Fix~$\mathsf{v} > 0$. We will prove that there exists~$\lambda < \lambda_c(\mathbb{T})$ such that the process~$(\dbh_t)_{t \ge 0}$ with parameters~$\mathsf{v}$ and~$\lambda$ produces a number of frozen herds that has expectation larger than one. By Lemma~\ref{lem_expectation_more}, this implies that~$\lambda > \bar{\lambda}(\mathsf{v})$.

	In what follows, we let~$\mathbb{E}_\lambda$ denote the expectation associated to a probability~$\mathbb{P}_\lambda$ under which the process~$(\dbh_t)_{t \ge 0}$ with parameters~$\mathsf{v}$ and~$\lambda$ is defined. The process is started from a single herd with a single marked particle (and no unmarked particles).	We let~$\xi_t = \cup_{i \in \mathcal{J}_t} \xi^i_t$, and recall that~$(\xi_t)_{t \ge 0}$ is a contact process on~$\mathbb{T}$ with parameter~$\lambda$.

	By combining the last four lemmas, for any~$\lambda < \lambda_c(\mathbb{T})$ we have
	\[\mathbb{E}_\lambda\left[\lim_{t \to \infty} F_5(\dbh_t)\right] \ge \mathsf{v}\cdot \frac{\lambda}{2+\lambda+\mathsf{v}}\cdot  \frac{\mathsf{v}}{2+\lambda \cdot d}\cdot  \frac{\lambda \cdot \left(1- \frac{1}{d-1}\right)}{2+\mathsf{v}+\lambda \cdot (d-1)} \cdot \int_0^\infty \mathbb{E}_\lambda [|\xi_t|]\;\mathrm{d}t . \]
Using~\eqref{eq_infinity_crit} and elementary continuity considerations, we have
	\[\lim_{\lambda \nearrow \lambda_c(\mathbb{T})} \int_0^\infty \mathbb{E}_\lambda[|\xi_t|]\;\mathrm{d}t = \infty.\]
	In particular, by taking~$\lambda$ close enough to~$\lambda_c(\mathbb{T})$ we obtain~$\mathbb{E}[\lim_{t \to \infty} F_5(\dbh_t)] > 1$.
\end{proof}

\textbf{Acknowledgements.} The research in this paper was funded by the grant NWO Physical Sciences TOP-Grant - Module 2 2016 EW, project number 613.001.603. The authors are thankful to NWO for the support.

\end{document}